\documentclass[reqno]{amsart}
\usepackage{amsmath,amsthm,amssymb}
\usepackage{lipsum} 
\usepackage[utf8]{inputenc}
\usepackage{esint}
\usepackage{amsfonts} 
\usepackage{xcolor}
\usepackage{tikz}
\usepackage[colorlinks=true]{hyperref}
\usepackage[a4paper, total={15cm, 23cm}]{geometry}
\usepackage{bbold}
\usepackage{enumerate}
\usepackage{stackengine}
\usepackage{calc}
\usepackage{amsaddr}
\usepackage[backend=biber, isbn=false, doi=false, url=false, maxnames=5] {biblatex}
\makeatletter
\numberwithin{equation}{section}
\theoremstyle{plain}
        \newtheorem{theorem}{Theorem}[section]
        
        \newtheorem{lemma}[theorem]{Lemma}
        \newtheorem{corollary}[theorem]{Corollary} 
        
        \newtheorem{definition}[theorem]{Definition} 
        \newtheorem{remark}[theorem]{Remark}  
        \newtheorem{example}[theorem]{Example}
         
        \newtheorem*{claim*}{Claim}

        \newtheorem*{fact*}{Fact} 
        
\newtheorem*{theorem*}{Theorem}
\newtheorem*{definition*}{Definition}
\newtheorem*{proposition*}{Proposition}
\usepackage{graphicx}
\let\oldmarginpar\marginpar
\renewcommand\marginpar[1]{\-\oldmarginpar[\raggedleft\footnotesize #1]
{\raggedright\footnotesize #1}}

\renewcommand\div{\text{div\,}}
\newcommand{\dist}{\mathrm{dist}\,}
\newcommand \loc {\text{loc}}

\newcommand{\R}{\mathbb{R}}

\newcommand{\N}{\mathbb{N}}

\newcommand\supp{\mathrm{supp}\,}
\renewcommand{\d}{\partial}

\newcommand{\dx}{\,dx}
\newcommand{\dt}{dt}
\newcommand{\Bop}{A}
\newcommand{\Bset}{E}
\newcommand{\phied}{\tilde{\phi}^{\epsilon,\delta}}

\newcommand\blfootnote[1]{%
  \begingroup
  \renewcommand\thefootnote{}\footnote{#1}%
  \addtocounter{footnote}{-1}%
  \endgroup
}

\allowdisplaybreaks

\title[Weak solutions of the Navier-Stokes equations with odd viscosity]{Weak solutions of the two-dimensional incompressible inhomogeneous Navier-Stokes equations in the presence of variable odd viscosity}

\author[R. Zimmermann]{Rebekka Zimmermann}

\begin{document}

\maketitle

\begin{abstract}
    We consider the two-dimensional incompressible inhomogeneous Navier-Stokes equations with odd viscosity, where the shear and the odd viscosity coefficients depend continuously on the unknown density function. 
    We establish the existence of weak solutions in both the evolutionary and stationary cases. 
    Furthermore, we investigate the limit of the weak solutions as the odd viscosity coefficient converges to a constant. 
    Lastly, we consider examples of stationary solutions for parallel, concentric and radial flows.
\end{abstract}


\section{Introduction}

This paper deals with 
a system of PDEs describing the motion of two-dimensional inhomogeneous incompressible fluids with odd viscosity. \textit{Odd} (or \textit{Hall}) viscosity is the anti-symmetric part of the viscosity tensor (in the sense of (\ref{viscositytensor}) below), and it is present in fluids with broken microscopic time-reversal symmetry and broken parity, see \cite{avron1998odd, avron1995viscosity}. We are going to discuss the physical model in Subsection \ref{odd:intro} below. 

The dynamics of such fluids can be described by the following generalized Navier-Stokes equations
\begin{equation}\label{odd}
   \left\{ \begin{array}{l}
      \d_t\rho+ \div(\rho u) =0,  \\
       \partial_t(\rho u) + \div(\rho u\otimes u) - \div(\mu_e(\nabla u + \nabla^Tu)) 
       - \div(\mu_o(\nabla u^\perp + \nabla^\perp u)) + \nabla \pi =\rho f,\\
       \div u=0, 
    \end{array}\right.
\end{equation}
where the scalar density function $\rho$, the velocity vector field $u = (u_1,u_2)^T$, and the scalar pressure $\pi$ are unknown and depend on the time and space variables $(t,x)\in [0,\infty)\times\Omega$. The domain $\Omega$ is either the whole space $\R^2$, or a bounded, open connected Lipschitz domain in $\R^2$. The variable shear and odd viscosity coefficients depend on the density function as follows:
\begin{align}\label{viscosity:coeff}
    \mu_e=\nu_e(\rho), \quad \mu_o=\nu_o(\rho),
\end{align}
for some given functions $\nu_e \in C(\R;[\mu_\ast,\mu^\ast])$, $\nu_o\in C(\R;[-\mu^\ast, \mu^\ast])$, with two positive constants $\mu^\ast,\mu_\ast>0$. The external force $f\in (L^2((0,\infty) \times\Omega))^2$ is given.

We denote 
\begin{align*}
    &u\otimes u = (u_iu_j)_{1\leq i,j\leq 2}, \quad \nabla u = (\partial_ju_i)_{1\leq i,j\leq 2}, \quad \nabla^T u = (\partial_i u_j)_{1\leq i,j\leq 2}, \\
    &x= \begin{pmatrix} x_1 \\ x_2 \end{pmatrix} , \quad \nabla = \begin{pmatrix} \partial_1 \\ \partial_2 \end{pmatrix}, \quad \nabla^\perp = \begin{pmatrix} -\partial_2 \\ \partial_1 \end{pmatrix}, \quad  u^\perp = \begin{pmatrix} -u_2 \\ u_1 \end{pmatrix},
\end{align*}
so that the strain tensors read as
\begin{align*}
    \nabla u + \nabla^T u 
    &= \begin{pmatrix} 2\d_1 u_1 & \d_2 u_1 + \d_1 u_2 \\ \d_2 u_1 + \d_1 u_2 & 2\d_2 u_2 \end{pmatrix}, \\
    \nabla u^\perp + \nabla^\perp u 
    &=\begin{pmatrix}-(\d_1 u_2+ \d_2 u_1)& \d_1 u_1-\d_2 u_2 \\ \d_1 u_1-\d_2 u_2 & (\d_1 u_2+\d_2 u_1) \end{pmatrix}.
\end{align*}

Equations (\ref{odd})$_1$ and (\ref{odd})$_2$ stand for the conservation of mass and the conservation of momentum, respectively, and (\ref{odd})$_3$ expresses the incompressibility of the fluid. 

The physics literature regarding odd viscosity in fluid dynamics is vast and has been an active area of research during the last thirty years. Despite of that, the notion of odd viscosity has as of yet received little attention from mathematicians. 

There is a large body of mathematical literature concerning the Navier-Stokes equations \eqref{odd} \textit{without} odd viscosity, i.e. $\mu_o\equiv0$:
\begin{equation}\label{NSintro}
   \left\{ \begin{array}{l}
      \d_t\rho+ \div(\rho u) =0,  \\
       \d_t(\rho u)+\div(\rho u\otimes u)-\div(\mu_e(\nabla u+\nabla^T u))+\nabla\pi=\rho f,\\
       \div u=0. 
    \end{array}\right.
\end{equation}
In particular, the existence of weak solutions of \eqref{NSintro} is well known in both the evolutionary and stationary cases \cite{he2020solvability, lions1996mathematical} (for the definition of weak solution we refer to Definitions \ref{weaksol} and \ref{def:weaksol,stat} below). 
It is natural to ask whether similar existence results hold if the odd viscosity is taken into account. 
In this article we give a positive answer to this question. 
More precisely, we establish the following facts:
\begin{enumerate}
    \item Weak solutions to \eqref{odd} exist in both the evolutionary and stationary cases (see Theorems \ref{main} and \ref{main:stat}).
    \item As the odd viscosity tends to some constant: $\mu_o\to\nu_o\equiv const$, any weak solution of \eqref{odd} which satisfies an energy inequality, converges (in a suitable sense) to a solution of the Naiver-Stokes equations without odd viscosity \eqref{NSintro} (see Corollaries \ref{rem:ev:conv} and \ref{rem:stat:conv}).
    \item For vanishing shear viscosity $\mu_e\equiv 0$, there do not in general exist stationary weak solutions of \eqref{odd} with $u\in H^1_{\mathrm{loc}}$ (see Example \ref{ex:radial}).
\end{enumerate}

In the following we review in more detail some known mathematical results concerning the existence theory of weak solutions to \eqref{NSintro} for the evolutionary and the stationary cases.


\subsection{The evolutionary Navier-Stokes system}

%
There has been extensive reseach on the Navier-Stokes equations \eqref{NSintro} without odd viscosity. 
When $\rho\equiv 1$ (and hence $\mu_e$ is a positive constant) the system turns into the classical homogeneous incompressible Navier-Stokes equation, which has been intensively studied. J. Leray \cite{leray1933etude} proved global-in-time existence of finite energy weak solutions in dimension $d=2,3$. For general $\rho$ and $\mu_e$ the existence of weak solutions was shown in P.-L. Lions' book \cite{lions1996mathematical}. There, P.-L. Lions proved the existence of weak solutions $(\rho,u)$ of (\ref{NSintro}) on various domains in the sense of Definition \ref{weaksol} below (with $\mu_o\equiv0$).

Here we are going to show that P.-L. Lion's results holds true even with the additional odd viscosity term. 
The main observation is that 
the \textit{a priori} energy balance, which is satisfied by smooth enough solutions, remains is preserved under odd viscosity. 
More precisely, the cancellation of the Frobenius inner product
\begin{align}\label{cancellation}
\begin{split}
    &(\nabla u^\perp + \nabla^\perp u): (\nabla u + \nabla^T u) 
    = \sum_{i,j=1}^2 (\nabla u^\perp + \nabla^\perp u)_{ij} (\nabla u + \nabla^T u)_{ij} \\
    &\quad = 2(\d_2u_2-\d_1u_1) (\d_1u_2+\d_2u_1) + 2(\d_1u_1 -\d_2u_2)(\d_2u_1 + \d_1u_2) \\
    &\quad =0
\end{split}
\end{align}
holds and hence, a priori, the energy equality holds when we test (\ref{odd})$_2$ by $u$:
\begin{align}\label{energy:eq}
\begin{split}
    &\int_\Omega \rho\vert u\vert^2\dx + \int_0^t\int_\Omega \mu_e \vert \nabla u +\nabla^Tu\vert^2\dx dt' \\
    &= \int_\Omega \rho_0 \vert u_0\vert^2 \dx + 2\int_0^t \int_\Omega \rho f\cdot u\dx dt', \quad \forall t>0 .
\end{split}
\end{align}
The energy \textit{in}equality is known to hold for P.-L. Lions' weak solutions (i.e. if $\mu_o\equiv 0$). The weak solutions we obtain for (\ref{odd}) satisfy the same energy inequality (see (\ref{energy:ineq}) below). This entails in particular that odd viscosity does not affect the energy of the fluid. 


\subsection{The stationary Navier-Stokes system}

The stationary counterpart of the equations (\ref{odd}) reads as follows:
\begin{equation}\label{odd:stat}
   \left\{ \begin{array}{l}
       \div(\rho u\otimes u) - \div(\mu_e(\nabla u + \nabla^Tu)) - \div(\mu_o(\nabla u^\perp + \nabla^\perp u)) + \nabla \pi = f,\\
       \div u=0,\; \div(\rho u) =0 , 
    \end{array}\right.
\end{equation}
%
For $\mu_o\equiv0$ previous results on (\ref{odd:stat}) include the following: If the flow is homogeneous, i.e. $\rho\equiv 1$ (and hence $\mu_e$ is a positive constant), then the system has been extensively studied 
and J. Leray \cite{leray1933etude} proved the existence of weak solutions $u\in (H^1(\Omega))^2$ on a simply connected domain $\Omega$ in $\R^2$. If the flow is inhomogeneous but $\mu_e>0$ is a constant, then N. N. Frolov \cite{frolov1993on} proved the existence and regularity of weak solutions of the form 
\begin{align}\label{frolov}
    (\rho,u)=(\eta(\phi), \nabla^\perp \phi) ,
\end{align}
where $\phi$ denotes the stream function of $u$, and $\eta$ is a H\"older continuous function. This Ansatz ensures that the divergence-free conditions (\ref{odd:stat})$_2$ are automatically satisfied: Formally we have 
\begin{align*}
    \div(u)&=\div(\nabla^\perp\phi)=0, \\
    \div(\rho u)&=\div(\eta(\phi)\nabla^\perp\phi) = \eta'(\phi)\nabla\phi \cdot \nabla^\perp \phi =0,
\end{align*}
and if $\phi\in H^2_{\mathrm{loc}}$, (\ref{odd:stat})$_2$ holds not only formally, but also in the sense of distributions. 

The H\"older continuity of $\eta$ in \cite{frolov1993on} was relaxed to continuity in \cite{santos2002stationary}. The existence result of weak solutions with \textit{variable} viscosity $\mu_e=\nu_e(\rho)$ was established by Z. He and X. Liao in \cite{he2020solvability}. The weak solutions there are also of Frolov's form (\ref{frolov}).

Assuming the form (\ref{frolov}) not only ensures that (\ref{odd:stat})$_2$ is satisfied, but it also has the advantage that the original problem (\ref{odd:stat}) for $(\rho,u)$ reduces to solving a fourth order elliptic equation for $\phi$. More precisely, applying the two-dimensional curl operator $\nabla^\perp \cdot$ to the momentum equation (\ref{odd:stat})$_1$ yields 
\begin{align*}
    \mathcal{L}\phi = -\nabla^\perp\cdot f + \nabla^\perp \cdot \div(\eta(\phi)\nabla^\perp\phi \otimes \nabla^\perp\phi) ,
\end{align*}
where the fourth order elliptic operator $\mathcal{L}$ is computed below in Section \ref{stationary}. It then remains to study the existence of weak solutions to this equation. The elliptic equation for the stream function has also been used in \cite{he2020solvability} with a different fourth order elliptic operator. The key observation is that the equation stays elliptic in the presence of odd viscosity. 

We are going to explain in this article how the proof from \cite{he2020solvability} can be modified to show a similar existence result in the presence of non-vanishing odd viscosity, $\mu_o\not\equiv 0$, in the momentum equation. In analogy to \cite{he2020solvability} we also consider examples of flows under certain symmetry assumptions on the density, and analyse the effects of the presence of odd viscosity on the velocity vector field.

\smallskip


In what follows, we give a brief explanation of the odd viscosity term in (\ref{odd})$_2$, and give an overview of previous physical results on this model. To the best of our knowledge there is much less mathematical research on hydrodynamic equations with odd viscosity, which is why this introduction focuses on the physical results.

\subsection{Odd viscosity: Properties and previous results}\label{odd:intro}

Viscosity of a fluid typically measures the resistance of the fluid to velocity gradients. It is expressed mathematically by a tensor $\eta_{ijkl}$ that acts as a coefficient of proportionality between viscous stress $\sigma_{ij}$ and strain rate $\partial_k u_l$, namely $\sigma_{ij}=\eta_{ijkl}\partial_k u_l$ \cite{landau1987fluid}. 
Symmetry under both parity and time-reversal are conditions which are satisfied by conventional fluids at thermal equilibrium, and in that case Onsager reciprocal relation \cite{onsager1931reciprocal} demands that $\eta_{ijkl}$ must be symmetric in the sense that $\eta_{ijkl}=\eta_{klij}$. Fluids in which time-reversal and parity are broken exhibit non-dissipative viscosity that is odd under each of these symmetries, which leads to the presence of an antisymmetric part in the viscosity tensor:
\begin{align}\label{viscositytensor}
    \eta_{ijkl}=\eta^e_{ijkl}+\eta^o_{ijkl}, \quad \text{with} \quad \eta^e_{ijkl} = \eta^e_{klij}, \quad \eta^o_{ijkl}=-\eta^o_{klij} .
\end{align}
In a two-dimensional incompressible isotropic fluid with broken parity and broken time-reversal, even and odd viscosity are specified by a single scalar $\mu_e$ and $\mu_o$, respectively, and hence the viscous stress tensor is characterized by two viscosity coefficients, one for the even part and one for the odd part \cite{avron1998odd, landau1987fluid} 
\begin{align*}
    \sigma =  \mu_e(\nabla u + \nabla^T u) + \mu_o(\nabla u^\perp +\nabla^\perp u) ,
\end{align*}
where $\mu_e$ is the kinematic shear (or ``even") viscosity, and $\mu_o$ the kinematic odd viscosity. When time-reversal and parity are broken, the viscosity tensor can have a non-vanishing odd part $\mu_o \not \equiv 0$, while it must vanish if at least one of these symmetries holds. 


The odd viscosity is quite different from the conventional shear viscosity, and significant odd viscosity may lead to counterintuitive properties of the fluid, which were illustrated with examples in \cite{avron1998odd}. For instance, odd viscosity can produce forces perpendicular to the direction of the fluid flow. 
%
An important difference between the two viscosity coefficients is that the shear viscosity $\mu_e$ is associated with dissipation, while the odd viscosity $\mu_o$ is of non-dissipative nature. 

The interest in odd viscosity was motivated by the seminal paper by J. E. Avron, R. Seiler and P. G. Zograf \cite{avron1995viscosity}, where it was shown that in general, quantum Hall states have non vanishing odd viscosity. Since then the role of odd viscosity in the context of quantum Hall fluids has been an active area of research, see \cite{abanov2013effective, bradlyn2012kubo, haldane2011geometrical, hoyos2014hall, hoyos2012hall, read2009non} and references therein. 
In classical fluids, odd viscosity can appear in various systems including polyatomic gases \cite{knaap1967heat}, 
chiral active fluids \cite{banerjee2017odd, soni2019odd, souslov2019topological}, magnetized plasmas \cite{pitaevskii2012physical} and fluids of vortices \cite{wiegmann2014anomalous}.
%
%
In three dimensions, odd terms in the viscosity tensor have been known for some time 
in hydrodynamic theories of superfluid He-3A \cite{vollhardt2013superfluid}.



In the following we will review some results on the constant and variable viscosity cases.

\subsubsection{Constant Viscosity}

When the density $\rho$ is positive and constant, say $\rho\equiv 1$, then by (\ref{viscosity:coeff}), $\mu_e>0$ and $\mu_o\in\R\setminus\{0\}$ are constants, and the system (\ref{odd}) turns into the homogeneous incompressible Navier-Stokes equations with odd viscosity
\begin{align}\label{odd:hom}
    \d_t u + u\cdot\nabla u - \mu_e\Delta u - \mu_o\Delta u^\perp + \nabla\pi=0, \quad \div u=0.
\end{align}
Much research has been done on related free-surface problems \cite{abanov2019free, abanov2018odd, doak2022nonlinear, monteiro2021nonlinear}. 
%

The authors in \cite{granero2022motion} considered equation (\ref{odd:hom}) with $\mu_e=0$ and a free-surface boundary, under the additional assumption that the fluid is irrotational: $\nabla^\perp \cdot u=0$. They developed asymptotic models of surface waves and proved well-posedness of these models in an analytic function space and also in a Sobolev setting.

In \cite{avron1998odd} J. E. Avron studied the odd viscosity effects in classical two-dimensional hydrodynamics by examining solutions to the wave equation as well as the homogeneous Navier-Stokes equations where the stress tensor is dominated by odd viscosity. These effects were shown to be subtle in the case when the fluid is incompressible, and they are most prominent in compressible flows. It has also been shown that if the fluid is almost incompressible, the odd viscosity effects are most visible at the dynamical boundary subject to no-stress or free-surface boundary conditions \cite{abanov2018odd, ganeshan2017odd}. Further investigations on observable effects of odd viscosity in classical two-dimensional incompressible hydrodynamics were conducted in \cite{banerjee2017odd, lapa2014swimming, lucas2014phenomenology}.
 

One important feature of the homogeneous incompressible Navier-Stokes equations with odd viscosity (\ref{odd:hom}) is that odd viscosity effects are absent when the fluid is spread on the entire plane or confined in rigid domains with no-slip boundary conditions. Indeed, due to the incompressibility of the fluid, the odd part of the viscosity tensor can be written as a gradient of the vorticity of the flow: $\Delta u^\perp =  -\nabla\omega$, where $\omega= \nabla^\perp \cdot u$, so that taking the two-dimensional curl yields the equation for the vorticity
\begin{align*}
    \partial_t\omega + u\cdot\nabla\omega = \mu_e\Delta\omega ,
\end{align*}
which is independent of $\mu_o$. Hence, vorticity is generated only by the symmetric part of the viscosity tensor, and as a consequence, the effects of odd viscosity on the flow of an incompressible fluid can come only through boundary conditions. For example, the signature of $\mu_o$ is present in surface waves and in the interface between two fluids governed by kinematic and no-stress boundary conditions, which explicitly depends on the odd viscosity \cite{ganeshan2017odd}. 

It is easy to see that for the same reason the systems (\ref{odd}) and (\ref{NSintro}) are equivalent if $\mu_o$ is a constant. Nevertheless, for density-dependent viscosity coefficient $\mu_o$ as in (\ref{odd}), it is in general not possible to write the odd viscosity term as a gradient, and $\mu_o$ might influence the flow. 


\subsubsection{Variable Viscosity}
In \cite{abanov2020hydrodynamics} the authors considered compressible fluids satisfying the strictly non-dissipative ($\mu_e\equiv0$) equations 
\begin{equation*}
   \left\{ \begin{array}{l}
      \d_t\rho+ \div(\rho u) =0,  \\
       \partial_t(\rho u) + \div(\rho u\otimes u) - \div(\nu_o\rho (\nabla u^\perp + \nabla^\perp u)) + \nabla (p(\rho)) =-\rho (B u^\perp + E),
    \end{array}\right.
\end{equation*}
where $B$ and $E=(E_1,E_2)$ are electromagnetic fields and $\nu_o$ is a real non-zero constant, i.e., equations (\ref{odd})$_{1,2}$ with $\mu_e\equiv0$ and $\mu_o = \nu_o\rho$ and an external force $f=Bu^\perp + E$ depending on $u$. These equations were paired with free-surface boundary conditions. The authors studied questions related to the variational formalism for the free-surface dynamics, linear surface waves and the incompressible limit.

The incompressible counterpart of the system in \cite{abanov2020hydrodynamics} was treated in \cite{fanelli2022well, fanelli2024effective} with the same choice for $\mu_o$ and no external force, and they proved local-in-time existence and uniqueness in the Sobolev setting $H^s(\R^2)$ for $s>2$, provided, among other assumptions, that the initial data satisfies $\rho_0-1\in H^{s+1}(\R^2)$ and $u_0\in (H^s(\R^2))^2$. These high regularity assumptions on the initial data were crucial for the authors to be able to exploit an underlying hyperbolic structure of (\ref{odd})$_2$ when $\mu_e\equiv0$. More precisely, working with a new set of variables 
\begin{align*}
    \omega = \nabla^\perp\cdot u, \quad \eta = \nabla^\perp\cdot (\rho u), \quad \theta = \eta-\Delta\rho ,
\end{align*}
allows one to resort to general theory of transport equations since the quantities $\omega$ and $\theta$ both satisfy simple transport equations transported by $H^s(\R^2)$ vector fields. It was also explained in \cite{fanelli2022well} about the challenges of obtaining a priori estimates for higher order derivatives when the shear viscosity term vanishes, due to the odd viscosity leading to a loss of derivatives. 

In a recent paper \cite{fanelli2024effective} the system with $\mu_e\equiv 0$ and $\mu_o=\nu_o\rho$ was reformulated into a system bearing strong similarities to the ideal magnetohydrodynamics equations. By use of this formulation the authors proved local well-posedness in a Besov space setting and provided a lower bound on the lifespan of the solution.

Here we will take $\mu_e\geq\mu_\ast>0$, and the presence of the shear viscosity term in (\ref{odd})$_2$ is crucial to obtain higher order estimates for the velocity $u$, if the density is only bounded, without any regularity assumptions.

\subsection{Organisation of the paper} The rest of this article is structured as follows. In Section \ref{main:result} we state the main existence result of weak solutions to the evolutionary and the stationary Navier-Stokes equations (\ref{odd}) and (\ref{odd:stat}), respectively. Section \ref{odd:proof} is concerned with the existence proof in the evolutionary case and Section \ref{stationary} with that of the stationary case. In Section \ref{stationary} we also consider examples of stationary parallel, concentric and radial flows.


\section{Main Results}\label{main:result}

In this section we give the definition of weak solutions of the systems (\ref{odd}) and (\ref{odd:stat}), and we state the main results of this paper.

\subsection{The evolutionary system}

We first fix a domain $\Omega\subset\R^2$ as in one of the following three cases. 
\begin{enumerate}
    \item \textit{(Dirichlet case)}: $\Omega$ is an open, bounded, connected domain in $\R^2$ with a Lipschitz boundary $\partial \Omega$. We look for solutions satisfying the non-slip boundary condition $u=0$ on $\partial\Omega$. 
    \item \textit{(Periodic case)}: $\Omega$ is a rectangle $(0,L_1)\times (0\times L_2)$, where $L_1,L_2>0$. We look for solutions $(\rho,u)$ which are periodic in $x_i$ of period $L_i$, $i=1,2$, respectively. 
    \item \textit{($\R^2$ case)}: $\Omega$ is the whole plane $\R^2$. We look for solutions $u\in L^\infty_{\mathrm{loc}}(0,\infty; (L^2(\R^2))^2)$. 
\end{enumerate}
We pose the following assumptions on the initial data: Let the initial density have a positive upper and lower bound
\begin{align}\label{rho:init}
    \rho_0\in L^\infty(\Omega), \quad 0< \rho_\ast\leq\rho_0\leq\rho^\ast \text{ a.e. in } \Omega,
\end{align}
and the initial velocity vector field
\begin{align}\label{u:init}
    u_0\in (L^2(\Omega))^2. 
\end{align}
Let $f\in (L^2((0,\infty)\times\Omega))^2$ be an external force. In the periodic case we consider $\rho_0$, $u_0$ and $f$ as functions defined on $\R^2$ by extending them periodically.

We define weak solutions for the evolutionary Navier-Stokes equations (\ref{odd}) in the following way.
\begin{definition}[Weak solutions of the evolutionary system]\label{weaksol}
We call a pair $(\rho,u)$ a weak solution of (\ref{odd}) with initial data given by (\ref{rho:init}) and (\ref{u:init}), if one of the following is satisfied for all $T>0$.
\begin{enumerate}
    \item \textit{(Dirichlet case)}: $\rho\in L^\infty((0,\infty)\times \Omega)\cap C([0,\infty);L^p(\Omega))$, $\forall p\in [1,\infty)$, $u\in L^2(0,T;(H_0^1(\Omega))^2)$;
    \item \textit{(Periodic case)}: $\rho\in L^\infty((0,\infty)\times \Omega)\cap C([0,\infty);L^p(\Omega))$, $\forall p\in [1,\infty)$, $u\in L^2(0,T;(H^1(\Omega))^2)$, $\rho$ and $u$ periodic in $x_i$ of period $L_i$, $i=1,2$;
    \item \textit{($\R^2$ case)}: $\rho\in L^\infty((0,\infty)\times \R^2)\cap C([0,\infty); L^p_{\mathrm{loc}}(\R^2))$, $\forall p\in [1,\infty)$, $u\in L^2(0,T;(H^1(\R^2))^2)$;
\end{enumerate}
and additionally there holds
\begin{align}
        &\div u=0, \quad \text{in the sense of distribution,} \label{uweak:div} \\
        &- \int_{\Omega} \rho_0\psi(0) \dx -\int_0^\infty \int_{\Omega} \rho\d_t \psi \dx\dt = \int_0^\infty \int_{\Omega} \rho u \cdot \nabla \psi \dx\dt , \label{rhoweak} \\
        &- \int_{\Omega} \rho_0u_0\cdot \varphi(0) \dx+\int_0^\infty \int_{\Omega} -\rho u\cdot \d_t\varphi -  \rho (u \otimes u) : \nabla\varphi \label{uweak} \\
        &+ \Bigl(\frac{\mu_e}{2}(\nabla u + \nabla^T u) + \frac{\mu_o}{2}(\nabla u^{\perp} + \nabla^\perp u)\Bigr) : (\nabla \varphi +\nabla^T\varphi) \dx\dt \nonumber \\
        &= \int_0^\infty\int_{\Omega}\rho f\cdot\varphi \dx\dt, \nonumber
    \end{align}
for all $\psi\in C_c^\infty ([0,\infty)\times \Omega)$, $\varphi\in (C_c^\infty ([0,\infty)\times \Omega))^2$, with $\div\varphi=0$, in the Dirichlet and whole plane case, or for all $\psi\in C_c^\infty([0,\infty);C^\infty(\Omega))$, $\varphi\in (C_c^\infty([0,\infty);C^\infty(\Omega)))^2$, with $\div \varphi=0$, and $\psi$, $\varphi$ being periodic in $x_i$ of period $L_i$, $i=1,2$, 
in the periodic case. Here the Frobenius inner product is defined by $A:B=\sum_{i,j=1}^2 a_{ij}b_{ij}$ for matrices $A=(a_{ij})_{i,j=1,2},B=(b_{ij})_{i,j=1,2}\in\R^{2\times 2}$.
\end{definition}

Our main result concerning the existence of weak solutions to the Navier-Stokes system (\ref{odd}) reads as follows.

\begin{theorem}[Existence of weak solutions of the evolutionary system]\label{main}
Let $\Omega$ be as in one of the three cases above, and let $\rho_0$ and $u_0$ be given by (\ref{rho:init}) and (\ref{u:init}), respectively. Then there exists at least one weak solution $(\rho,u)$ to the system (\ref{odd}) with initial data $\rho_0$ and $u_0$. Furthermore, this weak solution satisfies the energy inequality  
\begin{align}\label{energy:ineq}
\begin{split}
    &\int_\Omega \rho\vert u\vert^2\dx + \int_0^t\int_\Omega \mu_e \vert \nabla u +\nabla^Tu\vert^2\dx dt' \\
    &\leq \int_\Omega \rho_0 \vert u_0\vert^2 \dx + 2\int_0^t \int_\Omega \rho f\cdot u\dx dt', \quad \forall t>0 .
\end{split}
\end{align}
\end{theorem}

We follow a standard procedure to prove Theorem \ref{main} in Section \ref{odd:proof}. For the Dirichlet and periodic cases we use a Galerkin method as in \cite[Section 7]{feireisl2004dynamics}, and the whole plane case will be a consequence of the Dirichlet case for the derivation of which we follow the lines of \cite{lions1996mathematical}. Here we have to deal with additional second-order derivative terms of the velocity in the momentum equation due to the presence of odd viscosity. 


Let us make a few comments on the weak solutions obtained in Theorem \ref{main}.

\begin{remark}\label{rem:ev}
As in \cite{lions1996mathematical}, the results can be improved in the following sense:

The assumption $\rho_\ast>0$ in (\ref{rho:init}) can be relaxed to $\rho_\ast=0$, where vacuum in the fluid is allowed. This however requires additional assumptions on the initial data, and in Definition \ref{weaksol} the condition $u\in L^2(0,T;(H^1(\Omega))^2)$ is replaced by $\sqrt{\rho}u\in L^\infty(0,T;(L^2(\Omega))^2)$ and $\nabla u \in L^2((0,T)\times \Omega))^{2 \times 2}$, $T>0$.
    

The solutions given in Theorem \ref{main} are continuous in the sense that
\begin{align*}
    \rho u,\sqrt{\rho}u,u\in C([0,T];L^2_w(\Omega)), \quad \forall T>0 ,
\end{align*}
where $L^2_w(\Omega)$ denotes the space $L^2(\Omega)$ endowed with the weak topology. 
%
\end{remark}

The energy inequality (\ref{energy:ineq}) entails the following property of our weak solutions: let $\nu_e\in C(\R;[\mu_\ast, \mu^\ast])$ be fixed. We are interested in what happens to a weak solution in the limit $\nu_o\to c_0$ for some constant $c_0$, and are wondering whether the system (\ref{odd}) converges to (\ref{NSintro}) in some sense. The following corollary formulates this question more rigorously and gives a confirming answer.

\begin{corollary}\label{rem:ev:conv}
Let a sequence $(\nu_o^\epsilon)_{\epsilon\in (0,1)}$ of functions in $C(\R;[-\mu^\ast,\mu^\ast])$ be given, such that
\begin{align}\label{nuo:limit:cond}
    \Vert\nu_o^\epsilon-c_0\Vert_{C_b([\rho_\ast,\rho^\ast])}\to 0, \quad \epsilon\to 0,
\end{align}
for some constant $c_0\in [-\mu^\ast,\mu^\ast]$. For each $\epsilon \in (0,1)$ let $(\rho^\epsilon, u^\epsilon)$ denote a weak solution of (\ref{odd}) with odd viscosity coefficient $\nu_o^\epsilon$ which satisfies the energy inequality (\ref{energy:ineq}). 
Then there exists a function pair $(\rho,u)$ such that up to a subsequence
\begin{align}\label{main:limit:conv}
\begin{split}
    &\rho^{\epsilon}\to\rho, \quad \hspace*{2mm}\text{in } C([0,T];L^p(\Omega\cap B_R)), \quad \forall p\in [1,\infty),\, \forall T,R>0, \\
    &u^\epsilon\rightharpoonup^\ast u, \quad \text{in } L^\infty(0,T;L^2(\Omega))\cap L^2(0,T;H^1(\Omega)), \quad \forall T>0,
\end{split}
\end{align}
as $\epsilon\to 0$, and $(\rho,u)$ is a weak solution of (\ref{NSintro}), i.e. $(\rho,u)$ satisfies the conditions in Definition \ref{weaksol} with $\mu_o=0$.
\end{corollary}

Corollary \ref{rem:ev:conv} is proved in Subsection \ref{ev:conv:proof} below. The convergence ``(\ref{odd})$\to$(\ref{NSintro})" as $\nu_o\to c_0$ is compatible with the observation mentioned in the introduction, namely that for constant odd viscosity coefficient the odd viscosity term can be absorbed into the pressure and therefore does not affect the fluid flow (if the boundary conditions do not depend on $\mu_o$).

\subsection{The stationary system}

We next consider the stationary system (\ref{odd:stat}) on a bounded simply connected $C^{1,1}$ domain $\Omega$ in $\R^2$. We fix an external force $f\in (H^{-1}(\Omega))^2$, and the boundary value $g \in (H^{\frac{1}{2}}(\partial\Omega))^2$. In order for the boundary value condition $u|_\Omega=g$ to be compatible with the divergence-free condition for $u$, we need to postulate the zero flux condition
\begin{align}\label{zeroflux}
    \int_{\partial\Omega} g\cdot n\,d\sigma =0 ,
\end{align}
where $n=(n_1,n_2)$ denotes the outer normal vector to the boundary $\partial \Omega$. 

Firstly, we give the definition of a weak solution to the boundary value problem of the stationary Navier-Stokes equations (\ref{odd:stat}).

\begin{definition}[Weak solutions of the stationary system]\label{oddstat:weak:def}
We call a pair $(\rho,u)\in L^\infty(\Omega;[0,\infty))\times (H^1(\Omega))^2$ a weak solution to (\ref{odd:stat}) with boundary value $g\in (H^{\frac{1}{2}}(\partial\Omega))^2$ and external force $f\in (H^{-1}(\Omega))^2$, if 
\begin{align}
    &\div u=0, \quad \div(\rho u)=0, \quad \text{ in the sense of distribution}, \label{div:stat}\\
    &u|_{\partial\Omega} = g \quad \text{ in the trace sense}, \label{trace:stat}
\end{align}
and the integral identity
\begin{align}\label{uweak:stat}
\begin{split}
    &\int_\Omega \Bigl(\frac{\mu_e}{2}(\nabla u+\nabla^Tu) + \frac{\mu_o}{2}(\nabla u^\perp + \nabla^\perp u)\Bigr):(\nabla \varphi + \nabla^T\varphi)\dx \\
    &= \int_\Omega \rho(u\otimes u): \nabla\varphi \dx + \langle f, \varphi \rangle_{H^{-1}(\Omega)\times H_0^1(\Omega)}
\end{split}
\end{align}
holds for any $\varphi\in (H_0^1(\Omega))^2$ with $\div \varphi =0$.
\end{definition}

The existence and regularity result of weak solutions to the boundary value problem (\ref{odd:stat}) reads as follows.

\begin{theorem}[Existence and regularity of weak solutions of the stationary system]\label{main:stat}
Let $\Omega\subset\R^2$ be a bounded simply connected $C^{1,1}$ domain, and let the external force $f\in (H^{-1}(\Omega))^2$, and the boundary value $g \in (H^{\frac{1}{2}}(\partial\Omega))^2$ be given and satisfy the zero-flux condition (\ref{zeroflux}). 
\begin{enumerate}
    \item There exists at least one weak solution $(\rho,u)\in L^\infty(\Omega;[0,\infty))\times (H^1(\Omega))^2$ of the stationary Navier-Stokes equations (\ref{odd:stat}) which is of Frolov's form.
    \item Let $k\in\N$ and let $\Omega$ be a bounded simply connected $C^{k+1,1}$ domain. If $\nu_e\in C^k(\R;[\mu_\ast,\mu^\ast])$, $\nu_o\in C^k(\R;[-\mu^\ast, \mu^\ast])$, $\eta\in C^k(\R;[0,\rho^\ast])$, and the external force $f\in (H^{k-1}(\Omega))^2$ and the boundary value $g\in (H^{k+\frac{1}{2}}(\partial \Omega))^2$, then the weak solution $(\rho,u)$ from (1) satisfies $\rho\in H^{k}(\Omega)$ and $u\in (H^{k+1}(\Omega))^2$.
\end{enumerate}
\end{theorem}

We follow the arguments of \cite{he2020solvability} to prove Theorem \ref{main:stat} in Section \ref{stationary}. Since the presence of odd viscosity in the equation does not change much in the proof from \cite{he2020solvability}, we are only going to explain the main steps. The idea is to look for solutions which are of Frolov's form, i.e. $(\rho,u)=(\eta(\phi),\nabla^\perp\phi)$, for the stream function $\phi$ (determined by (\ref{phi:stat})), and some function $\eta\in L^\infty(\R;[0,\infty))$.

\begin{remark}\label{rem:stat}
Similarly to \cite{he2020solvability} one can show the existence of weak solutions on an exterior domain or on the whole plane $\R^2$. More precisely, if $\Omega$ is the exterior domain of a bounded simply connected $C^1$ set in $\R^2$, or $\Omega=\R^2$, then there exists a function pair $(\rho,u)\in L^\infty(\Omega;[0,\infty))\times (D^1(\Omega))^2$ satisfying (\ref{div:stat}), (\ref{trace:stat}) (if $\Omega\neq \R^2$) and (\ref{uweak:stat}). Here, the space $D^1(\Omega)$ is defined by
\begin{align*}
    D^1(\Omega) = \dot{H}^1(\Omega) \cap \Bigl( \cap_{n\in\N} H^1(\Omega\cap B_n)\Bigr) .
\end{align*}
\end{remark}

Moreover, one can prove a statement analogous to Corollary \ref{rem:ev:conv}.

\begin{corollary}\label{rem:stat:conv}
Suppose the sequence $(\nu_o^\epsilon)_{\epsilon\in (0,1)}$ is given as in Corollary \ref{rem:ev:conv}, and $(\rho^\epsilon,u^\epsilon)$ is any weak solution of (\ref{odd:stat}) of Frolov's form, i.e., $(\rho^\epsilon,u^\epsilon) = (\eta^\epsilon(\phi^\epsilon), \nabla^\perp\phi^\epsilon)$ for some function $\phi^\epsilon \in H^2(\Omega)$, and $\eta^\epsilon \in L^\infty(\R;[0,\rho^\ast])$. Then there exists a function pair $(\rho,u)$ such that up to a subsequence
\begin{align*}
    &\rho^\epsilon \rightharpoonup \rho, \quad \text{in } L^p(\Omega), \quad \forall p\in (1,\infty), \\
    &u^\epsilon \rightharpoonup u, \quad \text{in } H^1(\Omega),
\end{align*}
as $\epsilon\to 0$, and $(\rho,u)$ is a weak solution of the stationary system (\ref{odd:stat}) with $\mu_o\equiv 0$:
\begin{align}\label{NS:stat}
    \left\{ \begin{array}{l}
    \div(\rho u\otimes u) - \div(\mu_e(\nabla u + \nabla^Tu)) + \nabla \pi = f,\\
    \div u=0,\; \div(\rho u) =0 , \\
    u|_{\partial\Omega} = g,
    \end{array}\right.
\end{align}
i.e. $(\rho,u)$ satisfies the conditions in Definition \ref{oddstat:weak:def} with $\mu_o\equiv 0$.
\end{corollary}

The main step of the proof is to show that the sequence of stream functions $(\phi^\epsilon)_{\epsilon \in (0,1)}$ is bounded in $H^2(\Omega)$, by adjusting the arguments from \cite[Section 2.1]{he2020solvability}. This then yields the uniform boundedness of $(u^\epsilon)_{\epsilon\in (0,1)}$ in $H^1(\Omega)$, from which the assertion follows by standard compactness arguments.

\subsection{Notation} Throughout this paper we denote by $C$ a positive constant which may vary from time to time. For $v\in\R^2$ and a vector field $u\in (W^{1,1}_{\loc})^2$ we write $v\cdot\nabla u := (v\cdot\nabla)u= v_1\partial_1u+v_2\partial_2u$. For $p\in [1,\infty]$, $T>0$ and a Banach space $X$ we write $L^p_T(X) := L^p(0,T;X)$ and $L^p(X) := L^p(0,\infty;X)$, and we write $L^p$ or $H^m$ instead of $L^p(\Omega)$ or $H^m(\Omega)$, respectively, if the set $\Omega$ is clear from the context. The set $B_R=B_R(0)$ denotes the ball in $\R^2$ of radius $R>0$ centred at the origin.


\section{Proof of Theorem \ref{main}}\label{odd:proof}

In this section we prove the existence of weak solutions to the evolutionary system. We begin by regularising the given data in Subsection \ref{reg:data}. Afterwards we prove the existence and weak convergence of a sequence of solutions to the regularised system. To do so, we treat the Dirichlet and periodic cases (Subsection \ref{dirper}) separately from the whole plane case (Subsection \ref{whole}).

\subsection{Regularisation of the data}\label{reg:data}

In this subsection we regularise the initial data and external force as in \cite{lions1996mathematical}.

Let $\alpha\in C_c^\infty(\R^2)$ and $\beta\in C_c^\infty(\R)$ be smooth, compactly supported functions satisfying
\begin{align*}
    &\supp \alpha \subset B_1(0)\subset\R^2, \quad 0\leq\alpha \leq 1, \quad \int_{\R^2} \alpha \dx=1, \\
    &\supp \beta \subset (-1,1)\subset\R, \quad 0\leq\beta\leq 1, \quad \int_\R \beta \dt=1.
\end{align*}
Let $\epsilon\in (0,1)$. We set $\alpha_\epsilon= \epsilon^{-2} \alpha(\epsilon^{-1}\cdot)$, $\beta_\epsilon = \epsilon^{-1} \beta (\epsilon^{-1}\cdot)$. 
We choose approximate functions $\nu_e^\epsilon, \nu_o^\epsilon\in C^\infty ([0,\infty))$ with $\nu_e^\epsilon\geq\frac{\mu_\ast}{2}$, and 
\begin{align*}
    \Vert \nu_i^\epsilon - \nu_i \Vert_{C_b([\rho_\ast, \rho^\ast])} \leq \epsilon, \quad i\in\{e,o\} .
\end{align*}
Furthermore, in the Dirichlet case we set $\Omega_\delta=\{x\in\Omega : \dist(x,\partial\Omega)> \delta\}$, $\delta>0$. Then we define 
\begin{align*}
    &f^\epsilon= (1_{[\epsilon, \infty)} (f1_{\Omega_{2\epsilon}})\ast_x \alpha_\epsilon)\ast_t \beta_\epsilon, \quad \rho_0^\epsilon=(\rho_0 1_\Omega + 1_{\R^2 \setminus \Omega})\ast_x\alpha_\epsilon, \quad  \text{in the Dirichlet case,} \\
    &f^\epsilon= (1_{[\epsilon, \infty)} f\ast_x \alpha_\epsilon) \ast_t \beta_\epsilon, \quad \rho_0^\epsilon= \rho_0\ast_x\alpha_\epsilon, \quad \text{in the periodic and whole plane case,}
\end{align*}
and $u_0^\epsilon$ in the following way: Let $\overline{m}_0^\epsilon = \rho^\epsilon_0 ((u_01_{\Omega_{2\epsilon}}) \ast\alpha_\epsilon)$, and let $q_0^\epsilon\in C^\infty(\Omega)$ be a solution of
\begin{align*}
    &\div(\frac{1}{\rho_0^\epsilon}(\nabla q_0^\epsilon - \overline{m}_0^\epsilon))=0, \\
    &\frac{\partial q_0^\epsilon}{\partial n} =0 \quad \text{on } \partial \Omega, \quad \text{in the Dirichlet and periodic case,} \\
    &q_0^\epsilon \in H^1(\R^2), \quad \text{in the whole plane case,}
\end{align*}
where $n$ denotes the unit outward normal to $\partial\Omega$. Then we have $\overline{m}_0^\epsilon = \rho_0^\epsilon \overline{u}_0^\epsilon + \nabla q_0^\epsilon$ for some $\overline{u}_0^\epsilon\in C^\infty(\Omega)$ satisfying
\begin{align*}
    &\div \overline{u}_0^\epsilon =0, \\
    &\overline{u}_0^\epsilon\cdot n=0, \quad \text{on } \partial \Omega, \quad \text{in the Dirichlet and periodic case,} \\
    &\overline{u}_0^\epsilon \in L^2(\R^2), \quad \text{in the whole plane case.}
\end{align*}
Finally we choose
\begin{align*}
    u_0^\epsilon\in C_c^\infty(\Omega), \quad \div u_0^\epsilon =0, \quad \Vert u_0^\epsilon - \overline{u}_0^\epsilon \Vert_{L^2(\Omega)} \leq \epsilon.
\end{align*}
With the above definitions there holds
\begin{align*}
    &f^\epsilon\in C^\infty([0,\infty)\times\Omega), \quad \rho_0^\epsilon \in C^\infty(\Omega), \quad u_0^\epsilon \in C^\infty_c(\Omega), \\
    &f^\epsilon\to f, \quad \text{ in } (L^2((0,T)\times \Omega))^2, \quad \forall T>0, \\
    &\rho_0^\epsilon\to \rho_0, \quad \text{ in } L^p(\Omega\cap B_R),\quad \forall p\in [1,\infty), \, R>0, \\
    &m_0^\epsilon := \overline{m}_0^\epsilon + \rho^\epsilon_0 (u_0^\epsilon - \overline{u}_0^\epsilon) \to m_0 := \rho_0 u_0, \quad \text{in } L^2(\Omega),
\end{align*}
as $\epsilon\to 0$.

Since we use a fixed point argument to prove the existence of smooth solutions to the regularised system, we recall here the Schauder fixed point theorem (see e.g. \cite{gilbarg1977elliptic}). 
\begin{theorem}[Schauder fixed point theorem]\label{schauder:fp}
    Let $Y$ be a closed, convex set in a Banach space $X$, and let $F$ be a continuous mapping of $Y$ into itself such that $F(Y)$ is relatively compact. Then $F$ has a fixed point, i.e., there exists an element $x\in Y$ such that $F(x)=x$.
\end{theorem}

In the following we distinguish between the case that the domain $\Omega$ is bounded (Dirichlet and periodic case) and the case that $\Omega$ is unbounded (whole plane case).


\subsection{Dirichlet and Periodic case}\label{dirper}

We use Faedo-Galerkin approximations to construct a sequence of smooth solutions $(\rho^{\epsilon,n},u^{\epsilon,n})_{\epsilon\in (0,1), n\in\N}$ as in \cite{feireisl2004dynamics}. We start by proving their existence using the Schauder fixed point theorem. 

\subsubsection{Existence of regularised solutions}\label{loc:dp}

Let the set of divergence-free test functions on $\Omega$ be denoted by
\begin{align*}
    \mathcal{D}_\sigma&:= \{\varphi\in(C_c^\infty(\Omega))^2 : \div\varphi=0\}, \quad \text{ or} \\
    &:= \{\varphi\in (C^\infty(\Omega))^2 : \div\varphi =0, \, \text{ $\varphi$ is periodic in $x_i$ of period $L_i$, $i=1,2$}\},
\end{align*}
in the Dirichlet or periodic case, respectively. Furthermore, let
\begin{align*}
    L^2_\sigma &:= \{\varphi\in (L^2(\Omega))^2 : \div\varphi=0\}, \quad \text{ or} \\
    &:= \{\varphi\in (L^2(\Omega))^2 : \div\varphi=0, \, \text{ $\varphi$ is periodic in $x_i$ of period $L_i$, $i=1,2$}\}, \\
    V_\sigma &:= \{\varphi\in (H_0^1(\Omega))^2 : \div\varphi=0\}, \quad \text{ or} \\
    &:= \{\varphi\in (H^1(\Omega))^2 : \div\varphi=0, \, \text{ $\varphi$ is periodic in $x_i$ of period $L_i$, $i=1,2$}\},
\end{align*}
in the Dirichlet or periodic case, respectively. Since $\mathcal{D}_\sigma$ is dense in the Hilbert spaces ($L^2_\sigma$, $\Vert\cdot \Vert_{L^2}$) and ($V_\sigma$, $\Vert\cdot \Vert_{H^1}$), we can choose a countable system of functions $\{\varphi_n\}_{n\in\N} \subset \mathcal{D}_\sigma$ which form an orthonormal basis in $L^2_\sigma$ and are dense in $V_\sigma$. For $n\in\N$ let
\begin{align*}
    X_n := \mathrm{span}\{\varphi_1,...,\varphi_n \}.
\end{align*}
It is easy to see that $\overline{X_n}^{\Vert\cdot\Vert_{L^2}} = X_n$, so that $(X_n,\Vert\cdot\Vert_{X_n})$ is a Hilbert space endowed with the norm $\Vert\cdot \Vert_{X_n}:= \Vert\cdot \Vert_{L^2}$.

Notice that for $p\in [1,\infty]$ there exists a constant $C=C(n,\{\Vert\varphi_j\Vert_{W^{1,p}}\}_{j=1,...,n})$ such that
\begin{align}\label{xn:norm}
    \Vert\varphi\Vert_{W^{1,p}} \leq C \Vert\varphi \Vert_{X_n}, \quad \forall \varphi\in X_n .
\end{align}

We fix $T>0$ and $n\in\N$. Our first goal is to find for each $\epsilon \in (0,1)$ a function pair $(\rho^{\epsilon,n}, u^{\epsilon,n})\in C([0,T]\times\Omega)\times C([0,T];X_n)$ which satisfies (\ref{rhoweak}) and (\ref{uweak}) with test-functions $\phi\in C_c^\infty([0,T)\times\Omega)$ and $\varphi\in C^\infty_c([0,T);X_n)$, respectively, and initial data replaced by $\rho_0^\epsilon$, $u_0^{\epsilon}$, $\nu_i^\epsilon$, $i\in \{e,o\}$, and $f^\epsilon$.

Let us first introduce some notation. For a function $u\in C([0,T];X_n)$ we denote by $\rho[u]$ the unique solution of the transport equation
\begin{align*}
    \left\{ \begin{array}{l}
         \partial_t\rho + u\cdot\nabla\rho =0, \quad \text{in } (0,T)\times\Omega, \\
         \rho(0,\cdot) = \rho_0^\epsilon .
    \end{array}\right.
\end{align*}
By the smoothness of $u$ in space we know that $\rho[u],\partial_t \rho[u]\in C([0,T];C^k(\Omega))$ for any $k\in\N_0$, with
\begin{align*}
    &\Vert\rho[u](t)\Vert_{L^p} \leq \Vert\rho_0^\epsilon \Vert_{L^p}\exp\Bigl(C\int_0^t \Vert \nabla u(t') \Vert_{L^\infty} dt' \Bigr), \\
    &\Vert\rho[u](t)\Vert_{L^\infty} \leq \Vert\rho_0^\epsilon \Vert_{L^\infty} ,
\end{align*}
for $t\in [0,\infty)$ and $p\in [1,\infty)$. Clearly, $\rho[u]$ satisfies the weak formulation (\ref{rhoweak}) with velocity field $u$ and initial data $\rho_0^\epsilon$.

In order to reformulate equation (\ref{uweak}) as a fixed point problem we define for a function $\rho:\Omega\to [\rho_\ast,\rho^\ast]$ the operator
\begin{align*}
    M_\rho : X_n\to X_n^\ast, \quad M_\rho(\varphi) = \Bigl( X_n\ni \tilde{\varphi} \mapsto \int_\Omega \rho\varphi\cdot \tilde{\varphi}\dx\Bigr)
\end{align*}
where $X_n^\ast$ denotes the dual space of $X_n$. $M_\rho$ is linear, continuous and bijective with inverse given by
\begin{align*}
    M_\rho^{-1} : X_n^\ast \to X_n, \quad M_\rho^{-1} \varphi^\ast = \sum_{j=1}^n \varphi^\ast(\frac{1}{\rho}\varphi_j)\varphi_j ,
\end{align*}
and its operator norm satisfies the bound $\Vert M_\rho^{-1}\Vert_{X_n^\ast\to X_n} \leq \frac{n}{\rho_\ast}$.

\begin{lemma}\label{ex:ne}
For $\epsilon\in (0,1)$ and sufficiently large $n\in\N$ there exists $u^{\epsilon,n}\in C([0,T];X_n)\cap C^1((0,T];X_n)$ such that 
\begin{align}\label{uweak:ne:var}
\begin{split}
    &\int_\Omega \rho^{\epsilon,n}(t)u^{\epsilon,n}(t)\cdot\varphi(t) \dx - \int_\Omega \rho^\epsilon_0 u_0^{\epsilon}\cdot\varphi(0) \dx - \int_0^t\int_\Omega \rho^{\epsilon,n}u^{\epsilon,n}\cdot \partial_t \varphi \dx dt' \\
    &=\int_0^t \int_{\Omega} \rho^{\epsilon,n}(u^{\epsilon,n}\otimes u^{\epsilon,n}): \nabla\varphi - \frac{\mu_e^{\epsilon,n}}{2}(\nabla u^{\epsilon,n} + \nabla^T u^{\epsilon,n}):(\nabla \varphi + \nabla^T \varphi) \\
    &- \frac{\mu_o^{\epsilon,n}}{2}(\nabla u^{\epsilon,n \perp} + \nabla^\perp u^{\epsilon,n}):(\nabla \varphi + \nabla^T \varphi) + \rho^{\epsilon,n} f^\epsilon\cdot\varphi \dx dt'
\end{split}
\end{align}
for any function $\varphi \in C^1([0,T];X_n)$ and time $t\in[0,T]$, where $\rho^{\epsilon,n} = \rho[u^{\epsilon,n}]$, and $\mu_i^{\epsilon,n} = \nu_i^\epsilon(\rho^{\epsilon,n})$, $i\in \{e,o\}$. Moreover, $(\rho^{\epsilon,n}, u^{\epsilon,n})$ satisfies the energy equality
\begin{align}\label{energy:eq:ne}
\begin{split}
    &\frac{1}{2}\int_\Omega \rho^{\epsilon,n}(t) \vert u^{\epsilon,n}(t)\vert^2 \dx + \int_0^t \int_\Omega \frac{\mu_e^{\epsilon,n}}{2}\vert \nabla u^{\epsilon,n} + \nabla^T u^{\epsilon,n} \vert^2 \dx dt' \\
    &= \frac{1}{2} \int_\Omega \rho^{\epsilon,n}(0) \vert u^{\epsilon,n}(0)\vert^2 \dx + \int_0^t \int_\Omega \rho^{\epsilon,n} f^\epsilon\cdot u^{\epsilon,n} \dx dt'
\end{split}
\end{align}
for any $t\in [0,T]$.
\end{lemma}
\begin{proof}
We first reduce the problem of showing (\ref{uweak:ne:var}) to a simpler (but equivalent) problem: We claim that it suffices to look for $u^{\epsilon,n}$ satisfying
\begin{align}\label{uweak:ne}
\begin{split}
    &\int_\Omega \rho^{\epsilon,n}(t)u^{\epsilon,n}(t)\cdot \tilde{\varphi} - \rho_0^\epsilon u_0^{\epsilon}\cdot \tilde{\varphi} \dx \\
    &= \int_0^t \int_{\Omega} \rho^{\epsilon,n}(u^{\epsilon,n}\otimes u^{\epsilon,n}): \nabla\tilde{\varphi} - \frac{\mu_e^{\epsilon,n}}{2}(\nabla u^{\epsilon,n} + \nabla^T u^{\epsilon,n}):(\nabla \tilde{\varphi} + \nabla^T \tilde{\varphi}) \\
    &- \frac{\mu_o^{\epsilon,n}}{2}(\nabla u^{\epsilon,n \perp} + \nabla^\perp u^{\epsilon,n}):(\nabla \tilde{\varphi} + \nabla^T \tilde{\varphi}) + \rho^{\epsilon,n} f^\epsilon\cdot \tilde{\varphi} \dx dt'
\end{split}
\end{align}
for any $\tilde{\varphi}\in X_n$ and $t\in [0,T]$. Indeed, assume that (\ref{uweak:ne}) holds and take $\varphi \in C^1([0,T];X_n)$. Differentiating (\ref{uweak:ne}) with respect to time and then testing with $\tilde{\varphi}=\varphi(t)\in X_n$ for fixed $t\in [0,T]$ implies
\begin{align*}
    &\frac{d}{dt}\int_\Omega \rho^{\epsilon,n}u^{\epsilon,n} \cdot \varphi \dx - \int_\Omega \rho^{\epsilon,n} u^{\epsilon,n} \cdot \partial_t \varphi \dx \\
    &= \int_\Omega \rho^{\epsilon,n} (u^{\epsilon,n}\otimes u^{\epsilon, n}):\nabla\varphi - \frac{\mu_e^{\epsilon,n}}{2}(\nabla u^{\epsilon,n} + \nabla^T u^{\epsilon,n}):(\nabla \varphi + \nabla^T \varphi) \\
    &- \frac{\mu_o^{\epsilon,n}}{2}(\nabla u^{\epsilon,n \perp} + \nabla^\perp u^{\epsilon,n}):(\nabla \varphi + \nabla^T \varphi) + \rho^{\epsilon,n} f^\epsilon\cdot\varphi \dx,
\end{align*}
which yields the claimed equality (\ref{uweak:ne:var}) after integration in $t$.

In order to prove the existence of a solution to (\ref{uweak:ne}), we use the Schauder fixed point Theorem \ref{schauder:fp}. We begin by reformulating this equation into a fixed point problem. One can rewrite the equation as
\begin{align*}
    u^{\epsilon,n}(t) = M_{\rho[u^{\epsilon,n}](t)}^{-1}\Bigl((\rho_0^\epsilon u_0^{\epsilon})^\ast + \int_0^t \Bop(u^{\epsilon,n})(t')dt'\Bigr), \quad \forall t\in [0,T],
\end{align*}
where the element $(\rho_0^\epsilon u_0^{\epsilon})^\ast\in X_n^\ast$ and the operator $\Bop:C([0,T];X_n)\to L^\infty_T(X_n^\ast)$ are defined as
\begin{align*}
    (\rho_0^\epsilon u_0^{\epsilon})^\ast &= \Bigl(X_n\ni \tilde{\varphi} \mapsto \int_\Omega \rho_0^\epsilon u_0^{\epsilon} \cdot \tilde{\varphi} \dx\Bigr), \\
    \Bop(u) &= \Bigl(X_n\ni \tilde{\varphi} \mapsto \int_\Omega \rho[u] (u \otimes u) : \nabla \tilde{\varphi} - \frac{\mu_e^\epsilon}{2}(\nabla u + \nabla^T u):(\nabla\tilde{\varphi} + \nabla^T \tilde{\varphi}) \Bigr. \\
    &\Bigl.\quad- \frac{\mu_o^\epsilon}{2}(\nabla u^\perp + \nabla^\perp u):(\nabla \tilde{\varphi} + \nabla^T \tilde{\varphi}) + \rho[u] f^\epsilon\cdot \tilde{\varphi} \dx\Bigr),
\end{align*}
for $u\in C([0,T];X_n)$, with $\mu_i^\epsilon = \nu_i^\epsilon(\rho[u])$, $i\in\{e,o\}$. 
%

We define the set $E_T= \{u\in C([0,T];X_n) : \Vert u - u_0^{\epsilon} \Vert_{L^\infty_T(X_n)}\leq 1\}$, which is clearly convex in $C([0,T];X_n)$ (and non-empty for $n$ sufficiently large), and the operator
\begin{align*}
    F:\Bset_T\to C([0,T];X_n), \quad F(u)(t) := M_{\rho[u](t)}^{-1} \Bigl((\rho_0^\epsilon u_0^{\epsilon})^\ast + \int_0^t \Bop(u)(t')\, dt' \Bigr).
\end{align*}
We first show that $F$ has a fixed point provided $T=T(n)$ is sufficiently small.

\noindent\textbf{1. Continuity.} In order to verify that $F$ is continuous, let $u,v\in C([0,T];X_n)$ and $t\in [0,T]$. Writing
\begin{align*}
    F(u)(t)-F(v)(t) &= M_{\rho[u](t)}^{-1}\Bigl(\int_0^t \Bop(u)(t')-\Bop(v)(t')dt'\Bigr) \\
    &\quad+ (M_{\rho[u](t)}^{-1} - M_{\rho[v](t)}^{-1}) \Bigl((\rho_0^\epsilon u_0^{\epsilon, n})^\ast + \int_0^t \Bop(v)(t') \, dt'\Bigr),
\end{align*}
and using 
\begin{align}\label{Fcont:2}
\begin{split}
     \Vert M_{\rho[u](t)}^{-1} - M_{\rho[v](t)}^{-1} \Vert_{X_n^\ast\to X_n} 
     &= \Vert M_{\rho[v](t)}^{-1}(M_{\rho[v](t)} - M_{\rho[u](t)})M_{\rho[u](t)}^{-1} \Vert_{X_n^\ast\to X_n} \\
     &\leq C \Vert \rho [u](t) - \rho[v](t)\Vert_{L^\infty} \\
     &\leq C \Vert u-v \Vert_{L^\infty_T(X_n)} \exp(C \Vert v \Vert_{L^\infty_T(X_n)})
\end{split}
\end{align}
it is easy to see that
\begin{align*}
    \Vert F(u)-F(v) \Vert_{L^\infty_T(X_n)}
    &\leq C\Vert u-v \Vert_{L^\infty_T(X_n)}\exp(C\Vert v \Vert_{L^\infty_T(X_n)})( 1 + \Vert u \Vert_{L^\infty_T(X_n)} \\
    &\quad + \Vert v \Vert_{L^\infty_T(X_n)} + \Vert v \Vert_{L^\infty_T(X_n)}^2 + \Vert u_0^{\epsilon} \Vert_{L^2} + \Vert f^\epsilon\Vert_{L^2_T(L^2)})
\end{align*}
for some constant $C=C(n,T)$, which shows that $F$ is locally Lipschitz continuous and in particular continuous.

\noindent\textbf{2. Invariance of $\Bset_T$ under $F$.} Next we verify that $F$ maps the ball $E_T$ into itself provided the time $T$ is sufficiently small. Firstly, let $R>0$ such that $\Vert u \Vert_{L^\infty_T(X_n)}\leq R$ for any $u\in \Bset_T$. For $t\in(0,T]$ we have
\begin{align*}
    \Vert F(u)(t)-u_0^\epsilon\Vert_{X_n} &\leq C \Vert (\rho_0^\epsilon u_0^\epsilon)^\ast + \int_0^t A(u)(t')dt' - M_{\rho[u](t)}u_0^\epsilon\Vert_{X_n^\ast} \\
    &\leq C\Vert\rho_0^\epsilon - \rho[u](t)\Vert_{L^\infty} \Vert u_0^\epsilon \Vert_{L^2} + C \int_0^t\Vert u \Vert_{X_n}^2 + \Vert u \Vert_{X_n} + \Vert f^\epsilon\Vert_{L^2}dt',
\end{align*}
so that
\begin{align*}
    \Vert F(u)-u_0^\epsilon \Vert_{L^\infty_T(X_n)} \leq C\Vert \rho_0^\epsilon-\rho[u]\Vert_{L^\infty_T(L^\infty)} \Vert u_0^\epsilon \Vert_{L^2} + CT(R^2+R+\Vert f^\epsilon \Vert_{L^\infty_T(L^2)})
\end{align*}
for a constant $C=C(n)$ independent of $T$. Hence we can choose $T(n)\in (0,T]$ sufficiently small such that $F$ maps $\Bset_{T(n)}$ into itself.

\noindent\textbf{3. Relative compactness.} It is left to show that the image $F(\Bset_{T(n)})$ is relatively compact in the set $\Bset_{T(n)}$. To do so we use the version of the Arzela-Ascoli Theorem from \cite[Theorem 47.1]{munkres2000topology}: If $F(\Bset_{T(n)})$ is equicontinuous, and the sets $F_t =\{v(t) : v\in F(\Bset_{T(n)})\}$ are relatively compact in $X_n$ for all $t\in[0,T(n)]$, then it follows that $F(\Bset_{T(n)})$ is relatively compact in $\Bset_{T(n)}$. Notice that the latter condition is obvious here since $F_t\subset \{\varphi\in X_n : \Vert\varphi - u_0^{\epsilon} \Vert_{X_n}\leq 1 \}$ is a bounded set in the finite-dimensional vector space $X_n$. In order to show that $F(\Bset_{T(n)})$ is equicontinuous, let $u\in \Bset_{T(n)}$ and $t_1,t_2\in [0,T(n)]$. 
Writing
\begin{align*}
    F(u)(t_1)-F(u)(t_2) = &(M^{-1}_{\rho[u](t_1)}-M^{-1}_{\rho[u](t_2)}) \Bigl((\rho_0^\epsilon u_0^{\epsilon})^\ast + \int_0^{t_1} \Bop(u)(t')\, dt'\Bigr) \\
    &+ M^{-1}_{\rho[u](t_2)}\Bigl(\int_{t_2}^{t_1} \Bop(u)(t')\, dt'\Bigr),
\end{align*}
and using similar estimates as before, we deduce 
\begin{align*}
    \Vert F(u)(t_1)-F(u)(t_2) \Vert_{X_n} \leq C\vert t_1-t_2 \vert + C \vert t_1-t_2 \vert^{\frac{1}{2}}
\end{align*}
for some constant $C>0$ independent of $t_1,t_2$ and $u$, which yields equicontinuity.

\noindent\textbf{4. Conclusions.} By the Schauder fixed point Theorem \ref{schauder:fp} there exists a fixed point $u^{\epsilon,n} \in C([0,T(n)];X_n)$, i.e., $F(u^{\epsilon,n}) =u^{\epsilon,n}$. It is straightforward to verify that $F(u)$ belongs to $C^1((0,T];X_n)$ for any $u\in C([0,T];X_n)$, which implies that the fixed point is continuously differentiable in time: $u^{\epsilon,n} \in C^1((0,T(n)];X_n)$.

Differentiating (\ref{uweak:ne}) with respect to $t$, choosing $\varphi= u^{\epsilon,n}(t)$, and integrating by parts (which is justified by the smoothness of $\rho^{\epsilon,n}$ and $u^{\epsilon,n}$ in space, and where the boundary values in the periodic case vanish due to the periodicity) yields
\begin{align*}
    \frac{1}{2}\frac{d}{dt} \int_\Omega \rho^{\epsilon,n} \vert u^{\epsilon,n}\vert^2 \dx + \frac{\mu_\ast}{2} \int_\Omega \vert \nabla u^{\epsilon,n} + \nabla^T u^{\epsilon,n}\vert^2 \dx \leq \sqrt{\rho^\ast} \Vert f^\epsilon \Vert_{L^2(\Omega)} \Vert \sqrt{\rho^{\epsilon,n}} u^{\epsilon,n} \Vert_{L^2(\Omega)}
\end{align*}
for every $t\in [0,T(n)]$. Here we used the cancellation (\ref{cancellation}). By Gronwall's inequality it follows that,
\begin{align*}
    \Vert(\sqrt{\rho^{\epsilon,n}}u^{\epsilon,n})(t) \Vert^2_{L^2(\Omega)} &\leq \Vert\sqrt{\rho^\epsilon_0} u_0^\epsilon\Vert^2_{L^2(\Omega)} + 2\rho^\ast \Vert f^\epsilon \Vert_{L^2((0,t)\times\Omega)}^2 e^{2t} \\
    &\leq \rho^\ast \Vert u_0\Vert^2_{L^2(\Omega)} + 2 \rho^\ast \Vert f \Vert_{L^2((0,t)\times\Omega)}^2 e^{2t}
\end{align*}
for some constant $C>0$ independent of $\epsilon$ and $n$. Consequently, $u^{\epsilon,n}$ is uniformly bounded in $L^\infty_{T(n)}(L^2) \cap L^2_{T(n)}(H^1)$, with a bound independent of $\epsilon,n$ and $T(n)$. This allows us to iterate the above procedure to obtain a solution which is defined for all times up to $T$.

Finally, the energy equality (\ref{energy:eq:ne}) follows from choosing $\varphi= u^{\epsilon,n}$ in (\ref{uweak:ne:var}).
\end{proof}

\subsubsection{The limit $n\to\infty$}

We next show that for fixed $\epsilon\in (0,1)$ the sequence of weak solutions $(\rho^{\epsilon,n}, u^{\epsilon,n})_{n\in\N}$ constructed in Lemma \ref{ex:ne} converges weakly up to a subsequence to some function pair $(\rho^\epsilon, u^\epsilon)$ satisfying
\begin{align}\label{uweak:eps}
\begin{split}
    &\int_\Omega \rho^{\epsilon}(t)u^{\epsilon}(t)\cdot \varphi(t) \dx - \int_\Omega \rho_0^\epsilon u_0^\epsilon\cdot\varphi(0)\dx - \int_0^t\int_\Omega \rho^{\epsilon}u^{\epsilon}\cdot \partial_t \varphi \dx dt' \\
    &=\int_0^t \int_{\Omega} \rho^{\epsilon}(u^{\epsilon}\otimes u^{\epsilon}): \nabla\varphi - \frac{\mu_e^{\epsilon}}{2}(\nabla u^{\epsilon} + \nabla^T u^{\epsilon}):(\nabla \varphi + \nabla^T \varphi) \\
    &- \frac{\mu_o^{\epsilon}}{2}(\nabla u^{\epsilon \perp} + \nabla^\perp u^{\epsilon}):(\nabla \varphi + \nabla^T \varphi) + \rho^{\epsilon} f^\epsilon\cdot\varphi \dx dt'
\end{split}
\end{align}
for any function $\varphi \in C^1([0,T];V_\sigma)$ and time $t\in[0,T]$, where $\mu_i^\epsilon= \nu_i(\rho^\epsilon)$, $i\in\{e,o\}$.

Since $(u^{\epsilon,n})_{n\in\N}$ is uniformly bounded in $L^\infty_T(L^2) \cap L^2_T(H^1)$, there exists $u^\epsilon\in L^\infty_T(L^2) \cap L^2_T(H^1)$ ($u^\epsilon\in L^\infty_T(L^2) \cap L^2_T(H^1_0)$ in the Dirichlet case) such that up to a subsequence
\begin{align*}
    u^{\epsilon,n} \rightharpoonup^\ast u^\epsilon, \quad \text{in } L^\infty_T(L^2), \quad
    \nabla u^{\epsilon,n} \rightharpoonup \nabla u^\epsilon, \quad \text{in } L^2_T(L^2).
\end{align*}
Moreover, \cite[Theorem 2.4]{lions1996mathematical} yields the existence of $\rho^\epsilon\in L^\infty_T(L^\infty)$ such that up to a subsequence
\begin{align*}
    \rho^{\epsilon,n}\to \rho^\epsilon, \quad \text{in } C([0,T]; L^p), \quad \forall p\in [1,\infty)
\end{align*}
with $\rho^\epsilon$ being the solution to
\begin{align*}
    \partial_t\rho^\epsilon + \div(\rho^\epsilon u^\epsilon) =0, \quad \rho^\epsilon(0,\cdot) = \rho_0
\end{align*}
in the sense of distributions, i.e. (\ref{rhoweak}) holds with $\rho^\epsilon$. Consequently $\mu_i^{\epsilon,n}\to \mu_i^\epsilon$ in $C([0,T];L^p)$ for $p\in [1,\infty)$, $i\in\{e,o\}$. Since $\rho^{\epsilon,n}u^{\epsilon,n} \to \rho^\epsilon u^\epsilon$ in $L^2_T(L^r)$ for all $r\in [1,\infty)$ and $u^{\epsilon,n} \rightharpoonup^\ast u^\epsilon$ in $L^\infty_T(L^2)$ it follows that $\rho^{\epsilon,n}u^{\epsilon,n} \otimes u^{\epsilon,n} \rightharpoonup \rho^\epsilon u^\epsilon\otimes u^\epsilon$ in $L^2_T(L^p)$ for any $p\in [1,2)$. All of this together ensures the convergence of the integrals in (\ref{uweak:ne:var}) for fixed $\varphi\in C^1([0,T];X_{n_0})$ for any $n_0\in\N$. Since $\mathrm{span}\{\varphi_n\}_{n\in\N}$ is dense in $V_\sigma$, the integral identity also holds for any function $\varphi\in C^1([0,T];V_\sigma)$. 

Observe that $u^{\epsilon,n}(0)$ is given by the projection in $L^2$ of $u_0^\epsilon$ onto $X_n$. This implies that $u^{\epsilon,n}(0)\to u_0^\epsilon$ in $L^2$. With the same arguments as in \cite{lions1996mathematical} we can therefore take the limit $n\to\infty$ in the energy equality (\ref{energy:eq:ne}) to obtain the energy \textit{in}equality for $(\rho^\epsilon,u^\epsilon)$
\begin{align}\label{energy:ineq:e}
\begin{split}
    &\frac{1}{2}\int_\Omega \rho^{\epsilon}(t) \vert u^{\epsilon}(t)\vert^2 \dx + \int_0^t \int_\Omega \frac{\mu_e^{\epsilon}}{2}\vert \nabla u^{\epsilon} + \nabla^T u^{\epsilon} \vert^2 \dx dt' \\
    &\leq \frac{1}{2} \int_\Omega \rho^{\epsilon}(0) \vert u^{\epsilon}(0)\vert^2 \dx + \int_0^t \int_\Omega \rho^{\epsilon} f^\epsilon\cdot u^{\epsilon} \dx dt'
\end{split}
\end{align}
for all $t\in [0,T]$.

\subsubsection{The limit $\epsilon\to 0$}\label{limit:e}

Analogous arguments as before yield the existence of a pair of functions $(\rho,u) \in C([0,T];L^p)\times (L^\infty_T(L^2)\cap L^2_T(H^1))$ ($u\in L^\infty_T(L^2)\cap L^2_T(H^1_0)$ in the Dirichlet case) which is the weak limit of a subsequence of $(\rho^\epsilon,u^\epsilon)_{\epsilon\in (0,1)}$ and satisfies (\ref{rhoweak}) and
\begin{align*}
    &\int_\Omega \rho(t)u(t)\cdot\varphi(t) \dx - \int_\Omega \rho_0 u_0\cdot\varphi(0)\dx - \int_0^t\int_\Omega \rho u\cdot \partial_t \varphi \dx dt' \\
    &=\int_0^t \int_{\Omega} \rho(u\otimes u): \nabla\varphi - \frac{\mu_e}{2}(\nabla u + \nabla^T u):(\nabla \varphi + \nabla^T \varphi) \\
    &- \frac{\mu_o}{2}(\nabla u^{\perp} + \nabla^\perp u):(\nabla \varphi + \nabla^T \varphi) + \rho f\cdot\varphi \dx dt'
\end{align*}
for any $\varphi\in C^1([0,T];V_\sigma)$ for all $t\in [0,T]$, where $\mu_i = \nu_i(\rho)$, $i\in\{e,o\}$. Since $T>0$ was arbitrary we can find a weak solution $(\rho,u)$ which is defined for all times. This entails the weak formulations (\ref{rhoweak}) and (\ref{uweak}), and the existence part of Theorem \ref{main} for the Dirichlet and periodic case is proven.

By taking the limit $\epsilon\to 0$ in (\ref{energy:ineq:e}) we obtain the energy inequality for $(\rho,u)$.

\subsection{Whole plane case}\label{whole}

We now turn to the existence proof for the whole plane case $\Omega=\R^2$. Similarly to \cite{lions1996mathematical}, we derive the result for the whole plane case from the Dirichlet case.

Applying the Dirichlet case with $\Omega$ chosen as the ball $B_R:= B_R(0)$ of radius $R>0$, and initial data
\begin{align*}
    \rho_0^R:= \rho_0 1_{B_R}, \quad u_0^R:= u_0 1_{B_R}, \quad f^R := f 1_{B_R},
\end{align*}
yields a sequence $(\rho^R,u^R)_{R>0}$ of weak solutions on $B_R$ with
\begin{align*}
    &\rho^R\in L^\infty((0,\infty)\times B_R)\cap C([0,\infty); L^p(B_R)), \quad \forall p\in[1,\infty), \\
    &u^R \in L^2(0,T;(H_0^1(B_R))^2), \quad \forall T>0 .
\end{align*}
In the following we are going to show that up to a subsequence there holds
\begin{align*}
    \rho^R\to \rho \; \text{ in } C([0,T];L^p(B_M)), \qquad u^R \rightharpoonup u \; \text{ in } L^2(0,T;(H^1(B_M))^2) 
\end{align*}
as $R\to\infty$, for any $T,M>0$ and $p\in[1,\infty)$, where $(\rho,u)$ is a weak solution of (\ref{odd}) on the whole plane $\R^2$.

Indeed, by Remark \ref{rem:ev} the energy inequality (\ref{energy:ineq}) holds for the weak solutions $(\rho^R,u^R)$, which entails that the bound
\begin{align*}
    \Vert u^R\Vert_{L^2(0,T;H^1(\R^2))} \leq C_T, \qquad \Vert \sqrt{\rho^R} u^R \Vert_{L^\infty(0,T;L^2(\R^2))} \leq C_T
\end{align*}
holds uniformly in $R$, for all $T>0$. Here we view $\rho^R$ and $u^R$ as functions defined on $\R^2$ by extending them by zero onto $\R^2\setminus B_R$. This implies the existence of a function $u\in L^2_{\mathrm{loc}}(0,\infty;(H^1(\R^2))^2)$ such that  
\begin{align*}
    u^R\rightharpoonup^\ast u \; \text{ in } L^\infty(0,T;(L^2(\R^2))^2), \quad 
    \nabla u^R\rightharpoonup \nabla u \; \text{ in } L^2(0,T;(L^2(\R^2))^4), \quad \forall T>0,
\end{align*}
as $R\to\infty$, up to a subsequence. By \cite[Theorem 2.5]{lions1996mathematical} there exists a function $\rho \in L^\infty ((0,\infty)\times\R^2)$ satisfying the transport equation (\ref{odd})$_1$ in the weak sense with initial datum $\rho_0$ and velocity vector field $u$, and 
\begin{align*}
    \rho^R\to \rho \; \text{ in } C([0,T];L^p(B_M)), \quad \forall p\in [1,\infty),\, T,M>0
\end{align*}
as $R\to\infty$, up to a subsequence. Hence also $\mu_i^R:= \nu_i(\rho^R)\to\mu_i$ in $C([0,T];L^p(B_M))$ for any $p\in [1,\infty)$, $T,M>0$, and $i\in \{e,o\}$. Observe that also $\rho^R f^R\to \rho f$ in $L^2(0,T;(L^2(\R^2))^2)$ as $R\to\infty$. 

Notice that $\rho^R u^R\otimes u^R$ is bounded in $L^2(0,T;(L^{\frac{4}{3}}(\R^2))^4)$, and $\mu_e^R(\nabla u^R+\nabla^T u^R)$, $\mu_o^R(\nabla (u^R)^\perp + \nabla^\perp u^R)$ and $\rho^R f^R$ are bounded in $L^2((0,T)\times\R^2)$. Using the weak formulation (\ref{uweak}) we thus have for every $\varphi\in L^2(0,T;(H^2(\R^2))^2)$ with $\div\varphi=0$, 
\begin{align*}
    \vert\langle \partial_t(\rho^R u^R), \varphi \rangle \vert \leq C \Vert\varphi\Vert_{L^2(0,T;H^2(\R^2))} 
\end{align*}
for some constant $C>0$ independent of $R$. Since additionally $\rho^R \vert u^R\vert^2$ is bounded in $L^\infty(0,T;(L^1(\R^2))^2)$, it follows by \cite[Theorem 2.5]{lions1996mathematical} that 
\begin{align*}
    \sqrt{\rho^R}u^R \to \sqrt{\rho}u \; \text{ in } L^p(0,T;(L^r(B_M))^2), \quad \forall p\in (2,\infty), \; r\in [1, \frac{2p}{p-2}) ,
\end{align*}
for any $M>0$.
This ensures the convergence of the integrals in the weak formulation as $R\to\infty$, so that $(\rho,u)$ satisfies (\ref{uweak}) with $\Omega=\R^2$.


\subsection{Proof of Corollary \ref{rem:ev:conv}}\label{ev:conv:proof}

Let $(\nu_o^\epsilon)_{\epsilon\in (0,1)}$ denote the sequence from Corollary \ref{rem:ev:conv}, and let $(\rho^\epsilon, u^\epsilon)_{\epsilon\in (0,1)}$ be the corresponding sequence of weak solutions of (\ref{odd}) constructed in Theorem \ref{main}. The energy inequality (\ref{energy:ineq}) entails that $(u^\epsilon)_{\epsilon \in (0,1)}$ is uniformly bounded in $L^\infty(0,T;(L^2(\Omega))^2)\cap L^2(0,T;(\dot{H}^1(\Omega))^2)$ for every $T>0$.

Similar arguments as in Subsection \ref{whole} yield the existence of a function pair $(\rho,u)$ with
\begin{align*}
    \rho^\epsilon\to\rho \text{ in } C([0,T];L^p(B_M)), \quad u^\epsilon \rightharpoonup u \text{ in } L^2(0,T;(H^1(B_M))^2)
\end{align*}
as $\epsilon\to 0$, for any $T,M>0$ and $p\in [1,\infty)$, where $(\rho,u)$ is a weak solution of (\ref{odd}) with shear and odd viscosity coefficients $\mu_e$ and $c_0$, respectively. To prove that $(\rho,u)$ is in fact a weak solution of (\ref{NSintro}) it remains to verify the cancellation
\begin{align}\label{uvarphi:cancel}
    \int_\Omega (\nabla u^\perp + \nabla^\perp u): (\nabla \varphi + \nabla^T\varphi)\dx =0
\end{align}
for all $\varphi$ as in (\ref{uweak}). Indeed, using integration by parts twice and that $u$ and $\varphi$ are both divergence-free the left hand side can be written as
\begin{align*}
    &\int_\Omega 2(\d_1 u_1 -\d_2u_2)(\d_1\varphi_2 +\d_2 \varphi_1) - 2(\d_1u_2 + \d_2u_1)(\d_1\varphi_1 -\d_2\varphi_2)\dx \\
    &= \int_\Omega -2u_1(\d_{11}\varphi_2 + \d_{12}\varphi_1) + 2u_2 (\d_{12}\varphi_2 + \d_{22}\varphi_1) \\
    &\quad + 2u_2(\d_{11} \varphi_1 - \d_{12}\varphi_2) + 2u_1(\d_{12}\varphi_1 - \d_{22}\varphi_2)\dx \\
    &= \int_\Omega -2u_1\d_{11}\varphi_2 + 2u_2 \d_{22}\varphi_1 + 2u_2\d_{11}\varphi_1 - 2u_1\d_{22}\varphi_2 \dx \\
    &= \int_\Omega 2\d_1u_1\d_1\varphi_2 - 2\d_2u_2 \d_2 \varphi_1 + 2\d_2u_2\d_1 \varphi_2 - 2\d_1u_1\d_2\varphi_1 \dx \\
    &= \int_\Omega 2\d_1u_1(\d_1\varphi_2 + \d_2\varphi_1) - 2\d_1 u_1 (\d_1\varphi_2 + \d_2\varphi_1)\dx \\
    &=0 .
\end{align*}
This implies that the odd viscosity terms in the weak formulation (\ref{uweak}) vanishes and hence $(\rho,u)$ is a weak solution of (\ref{NSintro}).


\section{The Stationary Navier-Stokes equations}\label{stationary}

In this section we consider the stationary Navier-Stokes equation (\ref{odd:stat}). We first explain the main steps of the existence proof of Theorem \ref{main:stat}, and then study examples of the flow under certain symmetry assumptions on the density.

\subsection{Proof of Theorem \ref{main:stat}}

In this paragraph we explain the strategy of the proof of Theorem \ref{main:stat}. Since most arguments coincide with the ones in the proof of \cite[Theorem 1.5]{he2020solvability}, we omit the details of the proof and only describe the main ideas. We look for weak solutions which are of Frolov's form 
\begin{align}\label{frolov:sect4}
    (\rho,u)=(\eta(\phi), \nabla^\perp\phi),
\end{align}
for the stream function $\phi$, and some given function $\eta\in L^\infty(\R;[0,\infty))$. 

The existence proof of Theorem \ref{main:stat} is carried out in two steps: The first step is to formulate the boundary value problem for the stream function and to show the existence of a weak solution to this problem. In a second step one goes back to the original equation and shows that any pair which is of Frolov's form (\ref{frolov:sect4}) is indeed a weak solution to (\ref{odd:stat}). In the presence of odd viscosity the main changes are in step one since the boundary value problem for the stream function is modified. However, since the problem stays elliptic (due to $\mu_e\geq\mu_\ast>0$) the arguments from \cite{he2020solvability} still work.

The regularity results of Theorem \ref{main:stat} are proven using the elliptic equation (\ref{phi:stat}) for $\phi$, with the same arguments as in \cite{he2020solvability}, which we omit here.

\textbf{Step 1:} We begin by formulating the boundary value problem for the stream function. Firstly, we transform the equation for $u$ into a fourth order elliptic equation for the stream function $\phi$ by applying the two-dimensional curl operator $\nabla^\perp\cdot$ to the momentum equation (\ref{odd:stat})$_1$. A straightforward calculation yields
\begin{align*}
    \nabla^\perp\cdot \div(\mu_e(\nabla u + \nabla^Tu)) &= L_{\mu_e} \phi, \\
    \nabla^\perp\cdot \div(\mu_o(\nabla u^\perp + \nabla^\perp u)) &= A_{\mu_o}\phi,
\end{align*}
where the operators $L_{\mu_e}$ and $A_{\mu_o}$ are defined as
\begin{align*}
    L_{\mu_e} &= (\partial_{22}-\partial_{11})(\mu_e(\partial_{22}- \partial_{11})) + (2\partial_{12})(\mu_e(2\partial_{12})), \\
    A_{\mu_o} &= (\partial_{22}-\partial_{11})(\mu_o(2\partial_{12})) - (2\partial_{12})(\mu_o(\partial_{22}- \partial_{11})) .
\end{align*}
The momentum equation therefore becomes
\begin{align*}
    \mathcal{L} \phi = -\nabla^\perp\cdot f + \nabla^\perp\cdot \div(\eta(\phi) \nabla^\perp\phi\otimes\nabla^\perp\phi),
\end{align*}
where $\mathcal{L}= L_{\mu_e}+A_{\mu_o}$. Observe that $\mathcal{L}$ is an elliptic operator. Here we call an operator $\tilde{\mathcal{L}}=\sum_{\vert\alpha\vert,\vert\beta\vert\leq 2} D^\alpha (a_{\alpha\beta} D^\beta)$ elliptic if there exists some $\delta >0$ such that
\begin{align*}
    \delta\vert\xi\vert^2 \leq \sum_{\vert\alpha\vert= \vert\beta \vert=2} \mathrm{Re}\Bigl(a_{\alpha\beta}(x) \xi_\beta \xi_\alpha \Bigr) \leq \delta^{-1} \vert\xi\vert^2
\end{align*}
for almost every $x\in\Omega$ and every $\xi=(\xi_\alpha)_{\vert \alpha \vert=2}$, $\xi_\alpha \in \R$. 
We write
\begin{align*}
    L_{\mu_e} &= \partial_{11}(\mu_e\partial_{11}) + \partial_{22}(\mu_e \partial_{22}) - \partial_{11}((\mu_e-\frac{\mu_\ast}{2})\partial_{22}) - \partial_{22}((\mu_e-\frac{\mu_\ast}{2})\partial_{11}) \\
    &\quad + 2\partial_{12}((\mu_e-\frac{\mu_\ast}{2})\partial_{12}) + 2 \partial_{21}(\mu_e\partial_{21}) =: \sum_{\vert\alpha\vert= \vert\beta \vert=2} D^\alpha (a_{\alpha \beta}^e D^\beta), \\
    A_{\mu_o} &= \partial_{22}(\mu_o\partial_{12}) + \partial_{22}(\mu_o \partial_{21}) - \partial_{12}(\mu_o\partial_{22})-\partial_{21}(\mu_o \partial_{22}) - \partial_{11}(\mu_o\partial_{12}) - \partial_{11}(\mu_o \partial_{21}) \\
    &\quad + \partial_{12}(\mu_o\partial_{11}) + \partial_{21}(\mu_o \partial_{11}) =: \sum_{\vert\alpha\vert= \vert\beta \vert=2} D^\alpha (a_{\alpha \beta}^o D^\beta),
\end{align*}
so that for $\xi = (\xi_\alpha)_{\vert\alpha\vert=2}$, $\xi_\alpha\in \R$, there holds
\begin{align*}
    \frac{\mu_\ast}{2}\vert\xi\vert^2\leq  \sum_{\vert\alpha\vert= \vert\beta \vert=2} a_{\alpha\beta}^e(x)\xi_\alpha\xi_\beta\leq 2\mu^\ast \vert\xi\vert^2 , \quad 
    \sum_{\vert\alpha\vert= \vert\beta \vert=2} a_{\alpha \beta}^o (x) \xi_\alpha\xi_\beta =0
\end{align*}
for almost every $x\in\Omega$. Hence, the operator $\mathcal{L}$ with coefficients $a_{\alpha\beta}= a^e_{\alpha\beta} + a^o_{\alpha \beta}$ for $\vert\alpha\vert,\vert\beta\vert\leq 2$, is elliptic.

Let $\phi_0\in H^{\frac{3}{2}}(\partial\Omega)$ and $\phi_1\in H^{\frac{1}{2}}(\partial\Omega)$ be given functions. Pairing the equation for $\phi$ with boundary conditions we obtain the following boundary value problem for the stream function
\begin{equation}\label{phi:stat}
   \left\{ \begin{array}{l}
       \mathcal{L} \phi = -\nabla^\perp\cdot f + \nabla^\perp\cdot \div(\eta(\phi) \nabla^\perp\phi\otimes\nabla^\perp\phi), \\
       \phi |_{\partial\Omega} = \phi_0, \; \frac{\partial \phi}{\partial n} |_{\partial\Omega} = \phi_1 .
    \end{array}\right.
\end{equation} 
We define weak solutions of (\ref{phi:stat}) as follows.

\begin{definition}[Weak solutions of the elliptic equation]\label{def:weaksol,stat}
Let $\Omega\subset\R^2$ be a bounded simply connected $C^{1,1}$ domain, let $\eta\in L^\infty(\R;[0,\infty))$, and $\nu_e\in C(\R;[\mu_\ast, \mu^\ast])$, $\nu_o\in C(\R;[\mu_\ast, -\mu_\ast])$, $\mu_\ast,\mu^\ast >0$. Moreover, let $\phi_0\in H^{\frac{3}{2}}(\partial \Omega)$, $\phi_1\in H^{\frac{1}{2}}(\partial\Omega)$, and $f\in (H^{-1}(\R^2))^2$. We call $\phi\in H^2(\Omega)$ a weak solution to (\ref{phi:stat}) if
\begin{align*}
    \phi|_{\partial\Omega}=\phi_0, \quad \frac{\partial\phi}{\partial n} |_{\partial\Omega} = \phi_1, \quad \text{ in the trace sense, }
\end{align*}
and
\begin{align}\label{phiweak}
\begin{split}
    &\int_\Omega \mu_e((\partial_{22}\phi-\partial_{11}\phi)(\partial_{22} \psi-\partial_{11}\psi)+(2\partial_{12}\phi)(2\partial_{12}\psi))\dx \\
    &+\int_\Omega \mu_o((2\partial_{12}\phi)(\partial_{22}\psi-\partial_{11}\psi)-(\partial_{22}\phi-\partial_{11}\phi)(2\partial_{12}\psi))\dx \\
    &= \int_\Omega \rho(\nabla^\perp\phi\otimes\nabla^\perp\phi): \nabla \nabla^\perp\psi \dx + \langle f,\nabla^\perp\psi\rangle_{H^{-1}(\Omega)\times H_0^1(\Omega)}
\end{split}
\end{align}
for any $\psi\in H_0^2(\Omega)$, where $\rho=\eta(\phi)$ and $\mu_i = \nu_i(\rho)$, $i\in\{e,o\}$.
\end{definition}

We have the following result concerning the existence of weak solutions to the elliptic equation (\ref{phi:stat}).

\begin{lemma}
Let $\eta\in L^\infty(\R;[0,\infty))$, $\nu_e\in C(\R;[\mu_\ast, \mu^\ast])$, $\nu_o\in C(\R;[-\mu^\ast,\mu^\ast])$ and $f\in (H^{-1}(\Omega))^2$ be given. Then for any functions $\phi_0\in H^{\frac{3}{2}}(\partial \Omega)$, $\phi_1\in H^{\frac{1}{2}}(\partial\Omega)$ there exists a weak solution $\phi\in H^2(\Omega)$ to the boundary value problem (\ref{phi:stat}).
\end{lemma}

We are not going to give a proof of the preceding lemma here, but one can for example follow the lines of \cite[Section 2]{he2020solvability}; see also \cite{leray1933etude} for the classical stationary Navier-Stokes equations.

\textbf{Step 2:} We now go back to the original equation. Given a boundary value $g$, we show that for suitably chosen boundary values $\phi_0$ and $\phi_1$, a pair of Frolov's form (\ref{frolov:sect4}) is a weak solution of (\ref{odd:stat}) provided $\phi$ is a weak solution of (\ref{phi:stat}).

Notice that if $\phi\in H^2_{\mathrm{loc}}(\Omega)$ satisfies (\ref{phiweak}), then $(\rho,u) = (\eta(\phi),\nabla^\perp\phi)$ satisfies (\ref{uweak:stat}). Moreover, $\div(\rho u)=0$ holds in the distribution sense. Therefore, we have indeed proved that one can obtain a weak solution $(\rho,u)$ of (\ref{odd:stat}) from a weak solution $\phi$ of (\ref{phi:stat}) in the following way.

\begin{lemma}
Let $\eta\in L^\infty(\R;[0,\infty))$, $\nu_e\in C(\R;[\mu_\ast, \mu^\ast])$, $\nu_o\in C(\R;[-\mu^\ast,\mu^\ast])$ and $f\in (H^{-1}(\Omega))^2$ be given. Let $g\in (H^{\frac{1}{2}}(\partial \Omega))^2$ satisfy the zero-flux condition (\ref{zeroflux}) and let $C_0\in\R$. Moreover, let $\phi_0\in H^{\frac{3}{2}}(\partial \Omega)$ and $\phi_1\in H^{\frac{1}{2}}(\partial \Omega)$ be defined by
\begin{align*}
    \phi_0(\gamma(s))=-\int_0^s g\cdot n\,d\theta + C_0, \quad 
    \phi_1(\gamma(s)) = (u_0\cdot n^\perp)(\gamma(s)), 
\end{align*}
for $s\in [0,2\pi)$, where $\gamma:[0,2\pi)\to\partial\Omega$ is a parametrisation of the boundary $\partial\Omega$. If $\phi\in H^2(\Omega)$ is a weak solution of (\ref{phi:stat}), then the pair
\begin{align*}
    (\rho,u)=(\eta(\phi),\nabla^\perp \phi)
\end{align*}
is a weak solution of (\ref{odd:stat}).
\end{lemma}

The regularity results in Theorem \ref{main:stat} (2) follow from successively applying the differential operator $\d_j$, $j=1,2$, to the elliptic equation (\ref{phi:stat})$_1$ and using elliptic theory from \cite{dong2011higher} to deduce bounds on $\Vert\nabla^{k+2}\phi \Vert_{L^2(\Omega)}$. We omit the details and instead refer to \cite[Paragraph 1.3.2]{he2020solvability}.


\subsection{Proof of Corollary \ref{rem:stat:conv}}

Let $(\phi^\epsilon)_{\epsilon\in (0,1)}$ denote the weak solutions of (\ref{phi:stat}) corresponding to the odd viscosity coefficients $(\nu_o^\epsilon)_{\epsilon\in (0,1)}$. If we can show that the sequence of stream functions are uniformly bounded in $H^2(\Omega)$, then the sequence $(u^\epsilon)_{\epsilon \in (0,1)}$ is uniformly bounded in $(H^1(\Omega))^2$, and the claim follows by compactness arguments and the cancellation (\ref{uvarphi:cancel}).

To show the uniform boundedness of $(\phi^\epsilon)_{\epsilon \in (0,1)}$ in $H^2(\Omega)$ one can follow the lines of \cite[Paragraph 1.3.1]{he2020solvability}, which we are only giving a rough sketch of here.

By the inverse trace theorem and Whitney's extension theorem $\phi_0$ can be extended onto $\R^2$ (still denoted by $\phi_0$) such that $\phi_0\in H^2(\R^2)$ and $\frac{\d\phi_0}{\d n}|_{\d \Omega}=\phi_1$. Next we let $\delta\in (0,1)$ and $\zeta^\delta\in C_c^\infty(\R^2)$ be a smooth cut-off function satisfying $\zeta^\delta=1$ near the boundary $\d\Omega$, and $\zeta^\delta(x)=0$ for $\mathrm{dist}(x, \d\Omega)\geq \delta$. Let $\phi_0^\delta :=\phi_0\zeta^\delta$. Notice that $\phied := \phi^\epsilon - \phi_0^\delta \in H^2_0(\Omega)$, so that $\Vert\phied\Vert_{H^2(\Omega)} \sim \Vert\Delta\phied \Vert_{L^2(\Omega)}$.

The idea is to prove the uniform boundedness via a contradiction argument. Supposing that a subsequence $(\phi^{\epsilon_n})_{n\in\N}$ satisfies $\Vert \phi^{\epsilon_n} \Vert_{H^2(\Omega)}\to\infty$ as $n\to\infty$, we can test (\ref{phiweak}) by $\phied$ to derive an inequality of the form
\begin{align*}
    1\leq C\int_\Omega \eta(\phi^{\epsilon_n}) (\nabla^\perp g_n \otimes \nabla^\perp \phi_0^\delta):\nabla\nabla^\perp g_n \dx + \frac{C}{\Vert \tilde{\phi}^{\epsilon_n, \delta} \Vert_{H^2(\Omega)}},
\end{align*}
where $g_n= (\Vert \tilde{\phi}^{\epsilon_n,\delta} \Vert_{H^2(\Omega)})^{-1} \tilde{\phi}^{\epsilon_n,\delta}$. Using the uniform boundedness of the sequence $(g_n)_{n\in\N}$ in $H_0^2(\Omega)$ and compactness arguments one can show that the right hand side of the inequality converges to zero as $n\to\infty$ and $\delta\to 0$, which is a contradiction.

\subsection{Examples of parallel, concentric and radial flows}

In this paragraph we look at more concrete examples of solutions $(\rho,u)$ of the stationary Navier-Stokes system (\ref{odd:stat}) with external force $f=0$, under some symmetry assumptions on the density function, namely parallel, concentric and radial flows. The examples we consider are adapted from the case $\mu_o\equiv 0$ in \cite{he2020solvability}.

In the following we consider a weak solution $(\rho,u)$ of (\ref{odd:stat}) on $\R^2$. We write $\partial_j = \partial_{x_j}$, $j=1,2$, and make use of polar coordinates
\begin{align*}
    (x_1,x_2)=(r\cos\theta,r\sin\theta),
\end{align*}
so that
\begin{align*}
    \nabla_x = e_r\partial_r + \frac{e_\theta}{r}\partial_\theta, \quad \text{ where } \quad e_r=
    \begin{pmatrix}
        \frac{x_1}{r} \\ \frac{x_2}{r}
    \end{pmatrix}, \quad e_\theta =
    \begin{pmatrix}
        \frac{x_2}{r} \\ -\frac{x_1}{r}
    \end{pmatrix} .
\end{align*}

\begin{example}[Parallel flow]\label{ex:parallel}
Let $\rho=\rho(x_2)$ in $\R^2$ with $\partial_{2}\rho\neq 0$. The divergence-free condition $\div u=0$ and $\div(\rho u)=0$ imply that $u$ is of the form
\begin{align*}
    u=u_1(x_2)e_1,
\end{align*}
where $u_1$ is a function only depending on the variable $x_2$, and $e_1= (1\; 0)^T$ is the first standard basis vector in $\R^2$. Consequently, we see that
\begin{align*}
    &\div(\rho u\otimes u)=0,\\ 
    &\div(\mu_e(\nabla u+\nabla^Tu))= \partial_2(\mu_e\partial_2 u_1)e_1, \\ 
    &\div(\mu_o(\nabla u^\perp + \nabla^\perp u))= \partial_2(\mu_o\partial_2 u_1)e_2,
\end{align*}
where $e_2=(0\; 1)^T$. Hence, the system (\ref{odd:stat}) reads
\begin{align}\label{eq:parallel}
    \begin{pmatrix}
        -\partial_2(\mu_e\partial_2 u_1) + \partial_1\pi \\
        -\partial_2(\mu_o\partial_2 u_1) + \partial_2\pi
    \end{pmatrix}
    =
    \begin{pmatrix}
        0 \\ 0
    \end{pmatrix} .
\end{align}
The first line implies that $\pi$ is of the form $\pi(x_1,x_2) = x_1 \alpha(x_2) + \beta (x_1)$ for some functions $\alpha,\beta$, whereas the second line implies that $\partial_2 \pi$ must be independent of $x_1$. Therefore, $\alpha\equiv C$ for some constant $C\in\R$, which yields the two equations
\begin{align}\label{u:parallel}
    \partial_2(\mu_e\partial_2 u_1)=C , \quad \partial_2(\mu_o \partial_2u_1) =0 .
\end{align}
\end{example}

\begin{example}[Concentric flow]\label{ex:concentric}
Let $\rho=\rho(r)$ in $\R^2$ with $\partial_r\rho\neq 0$. The divergence-free condition $\div u=0$ and $\div(\rho u)=0$ imply that $u$ is of the form
\begin{align*}
    u=rg(r)e_\theta
\end{align*}
for some function $g$ only depending on $r$. It is straightforward to compute
\begin{align*}
    &\div(\rho u\otimes u)=-r\rho g^2 e_r, \\
    &\div(\mu_e(\nabla u+\nabla^Tu))=\frac{\partial_r(r^3\mu_e \partial_r g)}{r^2}e_\theta, \\
    &\div(\mu_o(\nabla u^\perp+\nabla^\perp u))= \frac{\partial_r(\mu_o r^3\partial_r g)}{r^2}e_r ,
\end{align*}
and thus the system (\ref{odd:stat}) reads
\begin{align}\label{eq:concentric}
    (-r\rho g^2-\frac{\partial_r(\mu_o r^3\partial_r g)}{r^2} + \partial_r\pi)e_r +(-\frac{\partial_r(r^3\mu_e \partial_r g)}{r^2} - \frac{1}{r}\partial_\theta\pi)e_\theta =0 .
\end{align}
By the linear independence of $e_r$ and $e_\theta$, each of the two bracket terms must be zero. The equation in $e_\theta$ direction yields $\pi(r,\theta)= \tilde{\alpha}(r) \theta +\tilde{\beta}(r)$ for some functions $\tilde{\alpha}, \tilde{\beta}$ depending only on $r$, and substituting $\partial_r\pi$ into the equation in $e_r$ direction yields
\begin{align}\label{u:concentric}
    \partial_r(r^3\mu_e\partial_r g)=Cr
\end{align}
for some constant $C\in\R$.
\end{example}


\begin{example}[Radial flow]\label{ex:radial}
Let $\rho=\rho(\theta)$ in $\R^2$ with $\partial_\theta\rho\neq 0$. The divergence-free condition $\div u=0$ and $\div(\rho u)=0$ imply that $u$ is of the form
\begin{align*}
    u=\frac{h(\theta)}{r}e_r
\end{align*}
for some function $h$ depending only on $\theta$. Consequently there holds
\begin{align*}
    &\div(\rho u\otimes u)=-\rho\frac{h^2}{r^3}e_r, \\
    &\div(\mu_e(\nabla u+\nabla^Tu)) = \frac{\partial_\theta(\mu_e \partial_\theta h)}{r^3}e_r - 2\frac{\partial_\theta(\mu_eh)}{r^3} e_\theta ,\\
    &\div(\mu_o(\nabla u^\perp+\nabla^\perp u))= -2\frac{\partial_\theta (\mu_oh)}{r^3}e_r - \frac{\partial_\theta(\mu_o\partial_\theta h)}{r^3}e_\theta ,
\end{align*}
and thus the system (\ref{odd:stat}) reads
\begin{align*}
    (-\rho\frac{h^2}{r^3}-\frac{\partial_\theta(\mu_e \partial_\theta h)}{r^3}+2\frac{\partial_\theta(\mu_oh)}{r^3}+\partial_r\pi)e_r + (2\frac{\partial_\theta(\mu_eh)}{r^3}+\frac{\partial_\theta(\mu_o \partial_\theta h)}{r^3}-\frac{1}{r}\partial_\theta\pi)e_\theta =0 .
\end{align*}
The equation in $e_\theta$ direction implies that
\begin{align*}
    \pi(r,\theta)=2\frac{\mu_eh}{r^2}+\frac{\mu_o\partial_\theta h}{r^2} + \hat{\alpha}(r)
\end{align*}
for some function $\hat{\alpha}$ depending only on $r$, and substituting $\partial_r\pi$ into the equation in $e_r$ direction yields
\begin{align}\label{ex:radial:eq}
    \rho h^2+\partial_\theta(\mu_e\partial_\theta h)-2\partial_\theta (\mu_o h) + 4\mu_e h + 2\mu_o\partial_\theta h=C
\end{align}
for some constant $C\in\R$.

If we take $\mu_e\equiv 0$, equation (\ref{ex:radial:eq}) becomes
\begin{align}\label{ex:radial:muo}
    C=\rho h^2 - 2\partial_\theta(\mu_o h)+2\mu_o\partial_\theta h .
\end{align}
This means that for a general domain $\Omega$ and general boundary value we can not expect $u\in H^1_{\mathrm{loc}}(\Omega)$ if $\mu_o$ has a jump. 
Indeed, let $\mu_o$ and $\rho$ be given by
\begin{align*}
    \mu_o= 1_{[0,\pi)} + 21_{[\pi,2\pi)}, \quad \rho = \nu_o^{-1}(1)1_{[0,\pi)} + \nu_o^{-1}(2)1_{[\pi,2\pi)} ,
\end{align*}
and assume there exists a solution $u\in H^1_{\mathrm{loc}}(\Omega)$, i.e. $h\in H^1_{\mathrm{loc}}([0,2\pi))$. We test (\ref{ex:radial:muo}) with $\varphi \in C_c^\infty((0,2\pi))$ to derive
\begin{align*}
    C\int_0^{2\pi}\varphi\, d\theta = \nu_o^{-1}(1) \int_0^\pi h^2 \varphi \, d\theta + \nu_o^{-1}(2) \int_\pi^{2\pi} h^2 \varphi \, d\theta - 2h(\pi)\varphi(\pi).
\end{align*}
If $h(\pi)\neq 0$ this implies
\begin{align*}
    \delta_\pi = \frac{\nu_o^{-1}(1)}{2h(\pi)} h^2 1_{[0,\pi)} + \frac{\nu_o^{-1}(2)}{2h(\pi)} h^2 1_{[\pi,2\pi)} - \frac{C}{2h(\pi)} ,
\end{align*}
where $\delta_\pi$ denotes the Dirac function. This is a contradiction since the right hand side is contained in $L^1_{\mathrm{loc}}((0,2\pi))$, whilst the left hand side is not. If $h(\pi)=0$, then
\begin{align*}
    \nu_o^{-1}(1)h^2 1_{[0,\pi)} + \nu_o^{-1}(2) h^2 1_{[\pi,2\pi)} = C \quad \text{in } \mathcal{D}'((0,2\pi)),
\end{align*}
which implies that $h^2$ must have a jump if $C\neq 0$, or $h=0$ if $C=0$. In the former case we obtain a contradiction since $h$ can not have a jump if it is contained in $H^1_{\mathrm{loc}}([0,2\pi))$. In the latter case $u$ can not satisfy the boundary condition for non-zero boundary value.


Consequently, it appears to be necessary to impose some regularity assumptions on $\mu_o$ in order to obtain the existence of solutions of (\ref{odd}) and (\ref{odd:stat}) if the shear viscosity in the momentum equation vanishes.
\end{example}





\section*{Acknowledgements}
The author would like to thank her Ph.D. supervisor JProf. Xian Liao for many helpful discussions and suggestions, and Dr. Francesco Fanelli for mentioning the interesting model to her during his stay at KIT.

\printbibliography

@article{fanelli2022well,
  title={Well-posedness theory for non-homogeneous incompressible fluids with odd viscosity},
  author={Fanelli, F. and Granero-Belinch{\'o}n, R. and Scrobogna, S.},
  journal={arXiv preprint arXiv:2211.15768},
  year={2022}
}

@article{ganeshan2017odd,
  title={Odd viscosity in two-dimensional incompressible fluids},
  author={Ganeshan, S. and Abanov, A. G.},
  journal={Physical review fluids},
  volume={2},
  number={9},
  pages={094101},
  year={2017},
  publisher={APS}
}

@article{abanov2018odd,
  title={Odd surface waves in two-dimensional incompressible fluids},
  author={Abanov, A. and Can, T. and Ganeshan, S.},
  journal={SciPost Physics},
  volume={5},
  number={1},
  pages={010},
  year={2018}
}

@article{avron1998odd,
  title={Odd viscosity},
  author={Avron, J. E.},
  journal={Journal of statistical physics},
  volume={92},
  pages={543--557},
  year={1998},
  publisher={Springer}
}

@book{lions1996mathematical,
  title={Mathematical Topics in Fluid Mechanics Oxford Lecture Series in Mathematics and its Applications 3},
  author={Lions, P.-L.},
  year={1996},
  publisher={Oxford University Press, New York}
}

@article{avron1995viscosity,
  title={Viscosity of quantum Hall fluids},
  author={Avron, J. E. and Seiler, R. and Zograf, P. G.},
  journal={Physical review letters},
  volume={75},
  number={4},
  pages={697--700},
  year={1995},
  publisher={APS}
}

@article{abanov2019free,
  title={Free-surface variational principle for an incompressible fluid with odd viscosity},
  author={Abanov, A. G. and Monteiro, G. M.},
  journal={Physical review letters},
  volume={122},
  number={15},
  pages={154501},
  year={2019},
  publisher={APS}
}

@article{abanov2020hydrodynamics,
  title={Hydrodynamics of two-dimensional compressible fluid with broken parity: variational principle and free surface dynamics in the absence of dissipation},
  author={Abanov, A. G. and Can, T. and Ganeshan, S. and Monteiro, G. M.},
  journal={Physical Review Fluids},
  volume={5},
  number={10},
  pages={104802},
  year={2020},
  publisher={APS}
}

@article{onsager1931reciprocal,
  title={Reciprocal relations in irreversible processes. I.},
  author={Onsager, L.},
  journal={Physical review},
  volume={37},
  number={4},
  pages={405--426},
  year={1931},
  publisher={APS}
}

@book{landau1987fluid,
  title={Fluid Mechanics: Course of Theoretical Physics, Volume 6},
  author={Landau, L. D. and Lifshitz, E. M.},
  volume={6},
  year={1987},
  publisher={(Butterworth-Heinemann, Oxford}
}

@article{doak2022nonlinear,
  title={Nonlinear shallow-water waves with vertical odd viscosity},
  author={Doak, A. and Baardink, G. and Milewski, P. A. and Souslov, A.},
  journal={SIAM Journal of Applied Mathematics},
  volume={83},
  number={3},
  pages={938--965},
  year={2023},
  publisher={SIAM}
}

@article{monteiro2021nonlinear,
  title={Nonlinear shallow water dynamics with odd viscosity},
  author={Monteiro, G. M. and Ganeshan, S.},
  journal={Physical Review Fluids},
  volume={6},
  number={9},
  pages={L092401},
  year={2021},
  publisher={APS}
}

@article{hoyos2014hall,
  title={Hall viscosity, topological states and effective theories},
  author={Hoyos, C.},
  journal={International Journal of Modern Physics B},
  volume={28},
  number={15},
  pages={1430007},
  year={2014},
  publisher={World Scientific}
}

@article{haldane2011geometrical,
  title={Geometrical description of the fractional quantum Hall effect},
  author={Haldane, F. D. M.},
  journal={Physical review letters},
  volume={107},
  number={11},
  pages={116801},
  year={2011},
  publisher={APS}
}

@article{abanov2013effective,
  title={On the effective hydrodynamics of the fractional quantum Hall effect},
  author={Abanov, A. G.},
  journal={Journal of Physics A: Mathematical and Theoretical},
  volume={46},
  number={29},
  pages={292001},
  year={2013},
  publisher={IOP Publishing}
}

@article{read2009non,
  title={Non-Abelian adiabatic statistics and Hall viscosity in quantum Hall states and $p_x+ i p_y$ paired superfluids},
  author={Read, N.},
  journal={Physical Review B},
  volume={79},
  number={4},
  pages={045308},
  year={2009},
  publisher={APS}
}

@article{hoyos2012hall,
  title={Hall viscosity and electromagnetic response},
  author={Hoyos, C. and Son, D. T.},
  journal={Physical review letters},
  volume={108},
  number={6},
  pages={066805},
  year={2012},
  publisher={APS}
}

@article{bradlyn2012kubo,
  title={Kubo formulas for viscosity: Hall viscosity, Ward identities, and the relation with conductivity},
  author={Bradlyn, B. and Goldstein, M. and Read, N.},
  journal={Physical Review B},
  volume={86},
  number={24},
  pages={245309},
  year={2012},
  publisher={APS}
}

@book{gilbarg1977elliptic,
  title={Elliptic partial differential equations of second order},
  author={Gilbarg, D. and Trudinger, N. S.},
  volume={224},
  number={2},
  year={1977},
  publisher={Springer}
}

@article{leray1933etude,
  title={{\'E}tude de diverses {\'e}quations int{\'e}grales non lin{\'e}aires et de quelques probl{\`e}mes que pose l'hydrodynamique},
  author={Leray, J.},
  journal={Journal de math{\'e}matiques pures et appliqu{\'e}es},
  volume={12},
  pages={1--82},
  year={1933}
}

@article{he2020solvability,
  title={Solvability of the two-dimensional stationary incompressible inhomogeneous Navier-Stokes equations with variable viscosity coefficient},
  author={He, Z. and Liao, X.},
  journal={arXiv preprint arXiv:2005.13277},
  year={2020}
}

@article{banerjee2017odd,
  title={Odd viscosity in chiral active fluids},
  author={Banerjee, D. and Souslov, A. and Abanov, A. G. and Vitelli, V.},
  journal={Nature communications},
  volume={8},
  number={1},
  pages={1573},
  year={2017},
  publisher={Nature Publishing Group UK London}
}

@article{souslov2019topological,
  title={Topological waves in fluids with odd viscosity},
  author={Souslov, A. and Dasbiswas, K. and Fruchart, M. and Vaikuntanathan, S. and Vitelli, V.},
  journal={Physical review letters},
  volume={122},
  number={12},
  pages={128001},
  year={2019},
  publisher={APS}
}

@article{soni2019odd,
  title={The odd free surface flows of a colloidal chiral fluid},
  author={Soni, V. and Bililign, E. S. and Magkiriadou, S. and Sacanna, S. and Bartolo, D. and Shelley, M. J. and Irvine, W. T. M.},
  journal={Nature physics},
  volume={15},
  number={11},
  pages={1188--1194},
  year={2019},
  publisher={Nature Publishing Group UK London}
}

@book{vollhardt2013superfluid,
  title={The superfluid phases of helium 3},
  author={Vollhardt, D. and Wolfle, P.},
  year={2013},
  publisher={Courier Corporation}
}

@article{lapa2014swimming,
  title={Swimming at low Reynolds number in fluids with odd, or Hall, viscosity},
  author={Lapa, M. F. and Hughes, T. L.},
  journal={Physical Review E},
  volume={89},
  number={4},
  pages={043019},
  year={2014},
  publisher={APS}
}

@article{lucas2014phenomenology,
  title={Phenomenology of nonrelativistic parity-violating hydrodynamics in 2+ 1 dimensions},
  author={Lucas, A. and Sur{\'o}wka, P.},
  journal={Physical Review E},
  volume={90},
  number={6},
  pages={063005},
  year={2014},
  publisher={APS}
}

@book{pitaevskii2012physical,
  title={Physical Kinetics: Volume 10},
  author={Pitaevskii, L. P. and Lifshitz, E. M.},
  volume={10},
  year={2012},
  publisher={Butterworth-Heinemann}
}

@article{wiegmann2014anomalous,
  title={Anomalous hydrodynamics of two-dimensional vortex fluids},
  author={Wiegmann, P. and Abanov, A. G.},
  journal={Physical review letters},
  volume={113},
  number={3},
  pages={034501},
  year={2014},
  publisher={APS}
}

@article{knaap1967heat,
  title={Heat conductivity and viscosity of a gas of non-spherical molecules in a magnetic field},
  author={Knaap, H. F. P. and Beenakker, J. J. M.},
  journal={Physica},
  volume={33},
  number={3},
  pages={643--670},
  year={1967},
  publisher={Elsevier}
}

@book{feireisl2004dynamics,
  title={Dynamics of viscous compressible fluids},
  author={Feireisl, E.},
  volume={26},
  year={2004},
  publisher={Oxford University Press}
}

@book{munkres2000topology,
  title={Topology, Second edition},
  author={Munkres, J. R.},
  year={2000},
  publisher={Pearson Education}
}

@article{granero2022motion,
  title={On the Motion of Gravity--Capillary Waves with Odd Viscosity},
  author={Granero-Belinch{\'o}n, R. and Ortega, A.},
  journal={Journal of Nonlinear Science},
  volume={32},
  number={3},
  pages={28},
  year={2022},
  publisher={Springer}
}

@article{frolov1993on,
  title={On the solvability of a boundary value problem of the motion of a nonhomogeneous fluid},
  author={Frolov, N. N.},
  journal={Mat. Zametki},
  volume={53},
  pages={130--140, 160},
  year={1993}
}

@article{santos2002stationary,
  title={Stationary solution of the Navier-Stokes equations in a 2d bounded domain for incompressible flow with discontinuous density},
  author={Santos, M. M.},
  journal={Zeitschrift f{\"u}r angewandte Mathematik und Physik ZAMP},
  volume={53},
  number={4},
  pages={661--675},
  year={2002},
  publisher={Springer}
}

@article{dong2011higher,
  title={Higher order elliptic and parabolic systems with variably partially BMO coefficients in regular and irregular domains},
  author={Dong, H. and Kim, D.},
  journal={Journal of Functional Analysis},
  volume={261},
  number={11},
  pages={3279--3327},
  year={2011},
  publisher={Elsevier}
}

@article{fanelli2024effective,
  title={Effective velocity and $L^\infty$-based well-posedness for incompressible fluids with odd viscosity},
  author={Fanelli, Francesco and Vasseur, Alexis F},
  journal={arXiv preprint arXiv:2401.17085},
  year={2024}
}

(R. Zimmermann) Institute for Analysis, Karlsruhe Institute of Technology, Englerstraße 2, 76131 Karlsruhe, Germany.

\textit{Email address:} {\tt rebekka.zimmermann@kit.edu}

\end{document}